\documentclass{article}

\PassOptionsToPackage{svgnames,dvipsnames}{xcolor}

\usepackage{amsfonts}
\usepackage{nicefrac}       
\usepackage{xfrac}
\usepackage{amsmath}
\usepackage{amsthm}
\usepackage{amssymb}
\usepackage{cancel}
\usepackage{mathtools}
\usepackage{bbm}
\usepackage{bm}

\usepackage{cases}         

\usepackage{mathalpha}
\usepackage{xspace} 

\usepackage{microtype}      
\usepackage{enumitem}

\usepackage[dvipsnames]{xcolor}

\usepackage{pifont}
\renewcommand{\checkmark}{\ding{51}}

\newcommand{\R}{\mathbb{R}}

\definecolor{lightgray}{RGB}{230, 230, 230}
\definecolor{mathred}{RGB}{204, 69, 90}
\definecolor{mathblue}{RGB}{4, 78, 112}
\definecolor{mathgreen}{RGB}{1, 135, 70}

\definecolor{linkcolor}{RGB}{0,120,130}

\usepackage[colorlinks=true,
    linkcolor=linkcolor,
    citecolor=linkcolor,
    filecolor=linkcolor,
    urlcolor=linkcolor,
    pagebackref=true]{hyperref} 
\usepackage{url}

\renewcommand*\backref[1]{\ifx#1\relax \else (Cit. on p. #1) \fi}

\usepackage[capitalize]{cleveref}

\Crefname{algorithm}{Algo.}{Algos.}
\Crefname{theorem}{Thm.}{Thms.}
\Crefname{lemma}{Lemma}{Lems.}
\Crefname{appendix}{Appx.}{Appx.}

\usepackage{amsthm}

\makeatletter
\newtheorem*{rep@theorem}{\rep@title}
\newcommand{\newreptheorem}[2]{%
  \newenvironment{rep#1}[1]{%
    \def\rep@title{#2 \ref{##1}}%
    \begin{rep@theorem}%
      \begingroup
      \let\label\@gobble 
    }{%
      \endgroup
    \end{rep@theorem}}
}
\makeatother

\newcommand{\repeq}[1]{\tag{\ref{#1}}} 

\theoremstyle{plain}
\newtheorem{theorem}{Theorem}

\newreptheorem{theorem}{Theorem}

\newtheorem{lemma}{Lemma}
\newreptheorem{lemma}{Lemma} 
\newtheorem{assumption}[theorem]{Assumption}

\theoremstyle{remark}

\theoremstyle{problem}

\usepackage[textsize=tiny]{todonotes}
\usepackage{mathtools}
\usepackage{multirow}
\usepackage{booktabs}
\usepackage{adjustbox}

\usepackage{caption}
\usepackage{subcaption}

\usepackage{tikz}
\usepackage{pgfplots}
\pgfplotsset{compat=1.17}
\usepackage{standalone}
\usetikzlibrary{shapes,arrows,fit,positioning, calc, decorations.markings}
\definecolor{captiongray}{RGB}{100,100,100}

\usepackage[toc,page,header]{appendix}
\usepackage[nohints]{minitoc}

\usepackage{etoc}

\usepackage{booktabs}
\usepackage{array}
\usepackage{makecell}

\usepackage{microtype}
\usepackage{graphicx}
\usepackage{subcaption}
\usepackage{booktabs} 
\usepackage{dblfloatfix}
\usepackage{enumitem}
\usepackage{hyperref}

\usepackage[accepted]{icml2026}

\usepackage{amsmath}
\usepackage{amssymb}
\usepackage{mathtools}
\usepackage{amsthm}

\usepackage{makecell}

\theoremstyle{plain}

\theoremstyle{definition}

\usepackage[textsize=tiny]{todonotes}

\icmltitlerunning{Accelerated and Stable Convergence with Anchored Generalized Optimistic Method}

\begin{document}

\twocolumn[
  \icmltitle{\fontsize{13}{25}\selectfont Accelerated and Stable Convergence with Anchored Generalized Optimistic Method}



  \icmlsetsymbol{equal}{*}

  \begin{icmlauthorlist}
    \icmlauthor{Motahareh Sohrabi}{1,2}
    \icmlauthor{Jianxin You}{1}
    \icmlauthor{Simon Lacoste-Julien}{1,2,4}
    \icmlauthor{Eduard Gorbunov}{3}
    \icmlauthor{Gauthier Gidel}{1,2,4}
  \end{icmlauthorlist}

  \icmlaffiliation{1}{Université de Montréal}
  \icmlaffiliation{2}{Mila - Quebec AI Institute}
  \icmlaffiliation{3}{Mohammed Bin Zayed University of Artificial Intelligence}
  \icmlaffiliation{4}{CIFAR AI Chair}

  \icmlcorrespondingauthor{Motahareh Sohrabi}{motahareh.sohrabi@mila.quebec}

  \icmlkeywords{Machine Learning, ICML}

  \vskip 0.3in
]



\printAffiliationsAndNotice{}  

\begin{abstract}
  We study first-order methods for solving monotone variational inequalities arising in min-max optimization. Classical approaches such as the extragradient method rely on two gradient queries per iteration, which limits their analysis and applicability in the online and stochastic settings. We propose a family of Generalized Optimistic Methods with Anchoring (GOMA), which combine two-time-scale optimistic updates with an anchoring term inspired by Halpern iteration. In the deterministic setting, GOMA achieves the optimal accelerated last-iterate rate $\mathcal{O}(1/k^2)$ on the squared gradient norm for monotone Lipschitz operators. In the stochastic setting with unbounded variance, a simplified single-call variant of GOMA achieves a last-iterate convergence rate of $\mathcal{O}(1/\sqrt{k})$ on the squared gradient norm. To the best of our knowledge, this is the first such guarantee for stochastic monotone Lipschitz variational inequalities in the unconstrained setting without variance reduction or growing batches.
\end{abstract}
\section{Introduction}

\looseness=-1
Minimax optimization and more generally, Variational Inequality (VI) problems, naturally arise in adversarial training \citep{goodfellow2014explaining, madry2017towards}, constrained optimization \citep{facchinei2003finite} and multi-agent reinforcement learning \citep{sidahmed2024addressing}, where the goal is to find equilibrium solutions under structured interaction of agents or competing objectives. When solving VIs classical gradient descent fails to converge even in simple bilinear games \citep{mertikopoulos2018optimistic}. A breakthrough came with the extragradient method (EG) of \citet{korpelevich1976extragradient}, which introduces a correction step and guarantees convergence under monotonicity. However, (i) EG requires two gradient evaluations per iteration, which is computationally expensive and makes it impractical in online or stochastic environments \citep{NEURIPS2020_eea5d933}. Moreover, subsequent work revealed that (ii) EG may fail in adversarial or stochastic regimes, motivating two-time-scale methods, such as DSEG algorithm \citep{NEURIPS2020_ba9a56ce}. Above all (iii) EG has a last-iterate convergence guarantee of $\mathcal{O}(1/k)$ for Monotone and Lipschitz operator in terms of squared operator norm, which is not optimal.

The optimistic method \citep{popov1980modification} reduces the per-iteration cost to a single gradient call by leveraging past gradients, alleviating the computational burden of EG. Generalized optimistic methods further improve robustness in adversarial and stochastic regimes through a two-time-scale design. Meanwhile, anchoring, inspired by the Halpern fixed-point iteration \citep{Halpern1967FixedPO, lieder2020convergence}, has emerged as an effective mechanism for accelerating VI algorithms and improving last-iterate convergence guarantees. We therefore introduce the Generalized Optimistic Method with Anchoring (GOMA), which combines these ideas to achieve low per-iteration complexity, robustness, and accelerated last-iterate convergence.

Our contributions are: 
\begin{itemize}[leftmargin=*]
    \item We introduce GOMA, combining two-time-scale optimistic updates with Halpern-type anchoring. 
    \item In the deterministic setting with monotone Lipschitz operators, we prove that GOMA attains an accelerated last-iterate convergence rate of $\mathcal{O}(1/k^2)$ in the squared operator norm, matching the complexity lower bound. 
    \item In the stochastic setting, we prove that a simplified variant of GOMA achieves a last-iterate convergence rate of $\mathcal{O}(1/\sqrt{k})$ on the squared operator norm under state-dependent noise. To the best of our knowledge, this is the first last-iterate guarantee in squared operator norm for stochastic monotone Lipschitz VIs in the unconstrained setting without variance reduction or growing batches.
\end{itemize}


\section{Preliminaries}
\label{sec:prelim}
Given a  vector field
$G:\mathbb{R}^d\to\mathbb{R}^d$, 
we study unconstrained \emph{variational inequality} (VI) problems defined as:
\begin{equation}
\label{eq:vi}
\text{find } x^\star\in\mathbb{R}^d \quad  \text{such that} \quad
G(x^\star)=0 .
\tag{VI}
\end{equation}
Throughout the paper, we measure convergence using the \emph{last-iterate squared residual}
$\|G(x_k)\|^2$.

\textbf{Assumptions.}
$G$ is \emph{monotone} and \emph{$L$-Lipschitz}:
\begin{align*}
\langle G(x)-G(y),\,x-y\rangle &\ge 0 , \qquad &&\forall x,y\in\mathbb{R}^d,\\
\|G(x)-G(y)\| &\le L\|x-y\| , &&\forall x,y\in\mathbb{R}^d .
\end{align*}
These assumptions characterize the standard class of monotone variational inequalities studied in first-order methods \citep{korpelevich1976extragradient,nemirovski2004prox}.

\textbf{Saddle-point problems.}
A central example of~\eqref{eq:vi} arises from saddle-point (min–max) optimization.
We consider
\[
\min_{x\in\mathbb{R}^d}\;\max_{y\in\mathbb{R}^d}\; f(x,y),
\]
where $f:\mathbb{R}^d\times\mathbb{R}^d\to\mathbb{R}$ is continuously differentiable.
Define $z=(x,y)\in\mathbb{R}^d\times\mathbb{R}^d$, and introduce the \emph{gradient operator}
\[
G(z)
\;:=\;
\begin{pmatrix}
\nabla_x f(x,y)\\[2pt]
-\nabla_y f(x,y)
\end{pmatrix}.
\]
Under the monotonicity of $G$ (equivalently, $f$ convex–concave), $z^\star = (x^\star, y^\star)$ is a saddle point if and only if $G(z^\star) = 0$. So, solving the saddle-point problem is equivalent to solving the variational inequality problem ~\eqref{eq:vi}.

In modern machine-learning applications of saddle-point problems, the variables $z=(x,y)$ parameterize neural networks and the saddle point lies in unbounded Euclidean space, making the squared operator norm $\|G(z_k)\|^2$ the de facto stationarity measure.
By contrast, classical algorithmic-game-theory settings with intrinsic bounded strategy spaces (e.g., matrix games on the simplex) use the gap function $\mathrm{GAP}(z) = \sup_{z'\in X}\langle G(z), z' - z\rangle$ \citep{nesterov2007dual} as the progress measure. The gap function requires a bounded domain, since the supremum diverges on $\mathbb{R}^d$, gives rise to fundamentally different analyses \citep{cai2022tight,abe2025boosting} and is suitable for constrained games' convergence analysis.

\section{Related Work}

\subsection{Algorithms for Solving Variational Inequalities} Solving variational inequalities \eqref{eq:vi}, with standard gradient descent can exhibit oscillatory behavior \citep{platt1988constrained, gidel2018variational}. A classical algorithm for addressing this behavior is the extragradient method \citep{korpelevich1976extragradient}, given by
\begin{equation}\tag{EG}\label{eq:eg}
\begin{aligned}
y_{k} &= x_k - \eta_k G (x_k) \\
x_{k+1} &= x_k - \eta_k G(y_{k}).
\end{aligned}
\end{equation}
For monotone Lipschitz operators, \eqref{eq:eg} achieves an \textit{ergodic} (average) rate of \(\mathcal{O}(1/k)\) duality gap \citep{nesterov2007dual}. This rate is optimal and matches the lower bound of \(\Omega(1/k)\)from \citet{nemirovski2004prox}.

However, a method may have an \textit{ergodic} convergence rate but no \textit{last-iterate} convergence: finite regret ensures convergence of the ergodic averages, while the iterates themselves may cycle indefinitely and not converge \cite{pmlr-v125-bailey20a}.

The last-iterate convergence rate of extragradient in terms of  squared operator norm for the same class of operators is \(\mathcal{O}(1/k)\) \citep{gorbunov2022extragradient}. This rate is not optimal, as the lower bound for this class is  \(\mathcal{O}(1/k^2)\) \citep{pmlr-v139-yoon21d}. As a result there exist several accelerated methods that achieve a rate of \(\mathcal{O}(1/k^2)\) \citep{yoon2024optimalaccelerationminimaxfixedpoint, lee2021fast, TranDinh2021HalpernTypeAA}. 

Despite their favorable convergence properties, extragradient methods rely on two operator evaluations per iteration. This is misaligned with the online-learning setting, which provides only a single gradient at the chosen action \cite{NEURIPS2020_eea5d933}. Optimistic gradient methods address this limitation by using extrapolation from past gradients rather than additional oracle queries \cite{NEURIPS2020_eea5d933,popov1980modification}.
\begin{equation}\tag{OM}\label{eq:og}
\begin{aligned}
y_{k} &= x_k - \eta_k G (y_{k-1}) \\
x_{k+1} &= x_k - \eta_k G(y_{k}).
\end{aligned}
\end{equation}
The optimistic gradient method achieves an \(\mathcal{O}(1/k)\) last-iterate convergence rate for monotone Lipschitz operators \citep{NEURIPS2022_893cd874,cai2022tight}. Several works have proposed ways to accelerate this rate, including \citet{TranDinh2021HalpernTypeAA, sedlmayer2023fast, cai2023accelerated}. Our work contributes to this line of research by studying \textit{last-iterate} acceleration for a class of \textit{optimistic gradient based} methods and providing stochastic convergence analysis.

\subsection{Two-Time-Scale Methods}
\looseness=-1

To improve last-iterate convergence and stability, two-time-scale strategies were introduced for extragradient-type dynamics in stochastic regimes where single-scale methods may fail to converge. In particular, \citet{NEURIPS2020_ba9a56ce} showed that, in the stochastic setting, using a larger step size for the extrapolation step than for the correction step prevents the failure of convergence of extragradient and yields almost sure last-iterate convergence at a rate of up to \(\mathcal{O}(1/k)\) in affine problems. 

Subsequently, \citet{lee2021fast} proposed the Fast Extragradient (FEG), which extends the two-time-scale idea to smooth problems under a negative comonotonicity assumption, achieving an accelerated \(\mathcal{O}(1/k^2)\) rate and stochastic convergence guarantees with growing batch-size.
These results show that time-scale decoupling is an effective mechanism for stabilizing and accelerating extragradient methods.

The same principle can be applied to optimistic, single-query methods. Accordingly, \citet{pmlr-v108-mokhtari20a} introduced the generalized optimistic method, which allows separate step sizes for the prediction and correction steps,
\begin{equation}\tag{Generalized OM}\label{eq:gog}
\begin{aligned}
y_{k} &= x_k - \gamma_k G (y_{k-1}), \\
x_{k+1} &= x_k - \eta_k G(y_{k}).
\end{aligned}
\end{equation}
In parallel, \citet{stooke2020responsive} and \citet{pmlr-v235-sohrabi24a} showed the effectiveness of proportional–integral (PI) controllers for solving the Lagrangian saddle-point formulation of constrained optimization problems. \citet{pmlr-v235-sohrabi24a} further showed that this PI-controller dynamics is equivalent to the generalized optimistic method of \citet{pmlr-v108-mokhtari20a}, revealing two-time-scale optimistic algorithms as an effective feedback-control system.

\subsection{Halpern-Type Acceleration}

A principal mechanism for \textit{accelerating} first-order methods in monotone variational inequalities is Halpern-type anchoring. The Halpern method \citep{Halpern1967FixedPO} was originally proposed for solving fixed-point problems
\[
z = T(z),
\]
where \(T:\mathbb{R}^d\to\mathbb{R}^d\) is a nonexpansive operator. Its classical iteration is
\begin{equation}
z_{k+1}
= \beta_k z_0 + (1-\beta_k)\,T(z_k),\nonumber
\end{equation}
where \(\beta_k\in(0,1)\) decreases to zero. To solve monotone variational inequalities \(G(z)=0\), with $G(z)$ is the gradient operator, a standard construction is to take
\(
T = (I+\alpha G)^{-1},
\)
the resolvent of \(G\), which is firmly nonexpansive when \(G\) is monotone~\citep{bauschke2020correction}. However, computing the resolvent is generally significantly more expensive than an explicit update using $G$ and therefore is outside the scope of this paper, which is focused on first-order methods.

Modern algorithms therefore use a \emph{Halpern-type anchoring} written directly in terms of the operator \(G\):
\begin{equation}\tag{Anchoring}\label{eq:halpern-type} 
\begin{aligned}
z_{k+1}
= z_k - \alpha_k G(z_k) + \beta_k (z_0 - z_k).
\end{aligned}
\end{equation}
This formulation can be viewed as a first-order realization of the classical Halpern iteration, as explained in \cite{pmlr-v125-diakonikolas20a} which is the anchoring mechanism used throughout this paper.

\textbf{Why anchoring helps.}
The anchoring update can be interpreted as a gradient step on a \emph{regularized} operator:
\[
\widetilde G_k(z) \;=\; G(z) + \tfrac{\beta_k}{\alpha_k}(z - z_0).
\]
Then \eqref{eq:halpern-type} can be written as a forward step with respect to \(\widetilde G_k\):
\[
z_{k+1}
= z_k - \alpha_k \widetilde G_k(z_k),
\]
where, \(\widetilde G_k\) is \(\tfrac{\beta_k}{\alpha_k}\)-strongly monotone, i.e.,
\[ 
\begin{aligned} 
&\langle \widetilde G_k(x)-\widetilde G_k(y),\,x-y\rangle \\
& = \langle G(x)-G(y),\,x-y\rangle + \tfrac{\beta_k}{\alpha_k}\|x-y\|^2 \;\ge\; \tfrac{\beta_k}{\alpha_k}\|x-y\|^2. \end{aligned} 
\]
Thus, the anchoring term endows the operator with an \emph{artificial strong monotonicity} that accelerates convergence. As \(\beta_k \downarrow 0\), this regularization vanishes, recovering the original monotone problem while providing acceleration in the transient regime.

\textbf{Connection to weight decay.} The regularizer $\tfrac{\beta_k}{\alpha_k}(z - z_0)$ in $\widetilde G_k$ is an $\ell_2$ (weight-decay) penalty \citep{
NIPS1991_8eefcfdf, loshchilov2018decoupled}, but centered at the initialization $z_0$ rather than the origin and applied with a vanishing coefficient. A constant weight-decay coefficient would shift the fixed point and bias the solution toward $z_0$; the decay $\beta_k \downarrow 0$ instead removes this bias asymptotically. Anchoring therefore provides the transient stabilization of weight decay, the artificial strong monotonicity noted above, while still converging to a solution $G(z^\star)=0$ of the original problem.

Anchoring mechanism has emerged as the key mechanism for last-iterate acceleration of variational inequalities. Methods like \textit{Extra Anchored Gradient (EAG)} algorithm \citep{pmlr-v139-yoon21d}, \textit{Fast Extragradient (FEG)} algorithm \citep{lee2021fast}, and \textit{Anchored Popov} \citep{TranDinh2021HalpernTypeAA} all use anchoring mechanism to achieve acceleration. The last-iterate behavior of anchored (simultaneous) gradient descent-ascent has likewise been studied \citep{ryu2019ode, surina2026improved}.

Anchoring has also been applied in reinforcement learning, where it provides stability and accelerates convergence. \citet{sokotaunified} shows that anchoring connects RL, quantal response equilibria, and zero-sum games by damping oscillations and guiding updates toward equilibria. More recently, anchoring has been used to accelerate value iteration \cite{NEURIPS2023_a8f2713b}, yielding faster convergence without sacrificing optimality. These results highlight anchoring as a general mechanism for stabilizing and accelerating learning in sequential decision-making.

\section{Generalized Optimistic Method with Anchoring (GOMA)}\label{sec:deterministic}

In this section, we introduce the \textit{Generalized Optimistic Method with Anchoring (GOMA)}, a family of algorithms that equips classical optimistic gradient method \eqref{eq:og} with two-time-scale update \eqref{eq:gog} and the anchoring mechanism  \eqref{eq:halpern-type}. 

\textbf{Generalized Optimistic Method with Anchoring:}
\begin{equation}\tag{GOMA}\label{eq:goma}
\begin{aligned}
y_{k} &= \beta_k x_0 + (1 - \beta_k)x_k - \gamma_k G (y_{k-1}), \\
x_{k+1} &= \beta_k x_0 + (1 - \beta_k)x_k - \eta_k G(y_{k}).
\end{aligned}
\end{equation}

Here, $\gamma_k$ and $\eta_k$ denote the step sizes for the exploration and update steps, respectively. The coefficient $\beta_k \in [0, 1)$ is the anchoring parameter, which gradually decays to zero as $k \to \infty$. Throughout this paper, we assume
$$\beta_k = \frac{a}{k + b} \quad \text{for} \quad b>a\geq0.$$
Several well-known algorithms arise as special cases of GOMA:
\begin{itemize}
    \item Setting $\beta_k = 0$ recovers the \textit{generalized optimistic method} \citep{pmlr-v108-mokhtari20a}.
    \item Setting $\gamma_k = \eta_k$ yields the \textit{anchored Popov algorithm} \citep{TranDinh2021HalpernTypeAA}.
    \item Setting both $\beta_k = 0$ and $\gamma_k = \eta_k$ reduces the scheme to \textit{classical Popov method} \citep{popov1980modification}, also known as the \textit{optimistic method} or \textit{past extragradient (PEG)}.
\end{itemize}

\subsection{Proof Outline}

To prove the convergence of \eqref{eq:goma}, we follow the standard potential-based analysis and construct the potential function \cref{eq:potential}, which will serve as the basis for the descent argument of $\|G(x_k)\|$.
\begin{equation}\label{eq:potential}
\begin{aligned}
     V_k = a_k\|G(x_k)\|^2 &+ b_k\langle G(x_k),\,x_k - x_0\rangle\\
     &+ c_kL^2\|x_k - y_{k-1}\|^2,
\end{aligned}
\end{equation}

where $a_k = c_k = \frac{b_k\eta_k}{2\beta_k}$, and $b_{k+1} = \frac{b_k}{1 - \beta_k}$.

To simplify hyperparameter tuning while preserving the benefits of a two-time-scale design, we consider two parameter setups:
(I) we fix the update step size to \(\eta_*\) and set the exploration step size to \(\gamma_k = (1 - \beta_k)\eta_*\); and
(II) we fix the exploration step size to \(\gamma_*\) and set the update step size to \(\eta_k = (1 - \beta_k)\gamma_*\).
Each setup may be preferable depending on whether a larger exploration step size or a larger update step size is desired.

\textbf{Case I: larger update step.}

We first study the schedule with a larger update step size, $\eta_k=\eta_\ast$ and exploration scaled by $(1-\beta_k)$. The next lemma states that the one-step potential function decrease holds whenever $(\eta_k,\beta_k)$ satisfy three elementary conditions that arise from bounding Lipschitz and monotonicity cross-terms.

\begin{lemma}[One-step potential decrease]\label{lemma:potential_eta} {\normalfont [Proof in \cref{appx:lemma1}.]}
Let $G$ be monotone and $L$-Lipschitz, and consider the iterates
\begin{equation}\tag{$\triangle$}\label{eq:goma_triangle}
\begin{aligned}
y_k &= \beta_kx_0 + (1-\beta_k)x_k - \eta_*(1-\beta_k) G(y_{k-1}), \\
x_{k+1} &= \beta_kx_0 + (1-\beta_k)x_k - \eta_* G(y_k).
\end{aligned}
\end{equation}

 We prove that the potential function \eqref{eq:potential} is decreasing if the step size $\eta_k$ satisfies the following conditions:
\begin{equation}
\eta_{k+1}\;\le\;\frac{\beta_{k+1}}{2M\,\eta_k\,\beta_k(1-\beta_k)},\label{eq:cond1}
\end{equation}
\begin{equation}
    1 - 2M\eta_k^2(1 - \beta_k)^2 - M\eta_k^2\beta_k^2 \ge 0,\label{eq:cond2}
\end{equation}
\begin{equation}
\hspace{-0.7em}\scalebox{0.85}{$\displaystyle
\eta_{k+1}
\le
\frac{\beta_{k+1}}{2\,\beta_k(1-\beta_k)}
\Bigl[\frac{2(1 - \beta_k^2) -4M\eta_k^2(1 - \beta_k)^2 -\beta_k^2}
{1 - 2M\eta_k^2(1 - \beta_k)^2 - M\eta_k^2\beta_k^2}\eta_k
\Bigr],$}
\label{eq:cond3}
\end{equation}
for $M := 2L^2(1+\theta)$ and $\theta \ge 0$. With $\widetilde c_{k+1}\ge0$ for $\theta=2$, these give:
$$V_k-V_{k+1}\ge \widetilde c_{k+1}L^2\|x_{k+1}-y_k\|^2.$$

Conditions of \cref{eq:cond1}, \cref{eq:cond2}, \cref{eq:cond3} are satisfied, with constant step size $\eta_k = \eta_* \in (0, \frac{1}{2\sqrt{3}L})$, and the choice of $\beta_k = \frac{2}{k+6}$.
\end{lemma}

\begin{theorem}\label{thm:main}{\normalfont [Proof in \cref{appx:thm:main}.]}
Suppose $G$ is monotone and $L$-Lipschitz continuous. Consider the updates of \cref{eq:goma_triangle}
With the parameter choices
\[
\beta_k = \frac{2}{k+6}, 
\quad 
\eta_* \in \Bigl(0, \tfrac{1}{2\sqrt{3}L}\Bigr),
\]
\[
\begin{aligned}
a_k = c_k
&= \frac{b_0\,\eta_*}{80}\,(k+4)(k+5)(k+6),\\
b_k
&= \frac{b_0}{20}\,(k+4)(k+5)
\end{aligned}
\]
the potential function's decrease from \cref{lemma:potential_eta} implies the bound
\begin{equation}\label{eq:simple_bound}
\|G(x_k)\|^2 
\;\le\; \frac{16/\eta_*^2 + 72L^2}{(k+6)^2}\,\|x_0 - x^\star\|^2 .
\end{equation}
Moreover, the constant $16/\eta_*^2+72L^2$ is decreasing in $\eta_*$, so it is smallest at the largest step; as $\eta_*\to\frac{1}{2\sqrt{3}L}$ it gives
\begin{align}\label{eq:edge_bound1}
\|G(x_k)\|^2 
&\le \frac{264L^2}{(k+6)^2}\,\|x_0-x^\star\|^2.
\end{align}
\end{theorem}

The bound \eqref{eq:simple_bound} gives an $\mathcal{O}(1/k^2)$ decay of the residual $\|G(x_k)\|^2$ with explicit constants and a single scalar hyperparameter $\eta_\ast$.

\textbf{Case II: larger exploration step.}
We next analyze the complementary schedule with a larger exploration step size, keeping the update scaled by $(1-\beta_k)$. This variant can be preferable when exploration requires a larger look-ahead while updates must remain conservative.

\begin{lemma}\label{lemma:potential_gamma}{\normalfont [Proof in \cref{appx:potential_gamma}.]}
Let $G$ be monotone and $L$-Lipschitz, and consider the iterates
\begin{equation}\tag{$\square$}\label{eq:goma_square}
\begin{aligned}
y_k &= \beta_k x_0 + (1-\beta_k)x_k - \gamma_*\, G(y_{k-1}), \\
x_{k+1} &= \beta_k x_0 + (1-\beta_k)x_k - (1-\beta_k)\gamma_*\, G(y_k).
\end{aligned}
\end{equation}

The potential in \eqref{eq:potential} is non-increasing for the algorithm if the following conditions are satisfied:
\begin{align}
\gamma_{k+1}
&\;\le\;
\frac{\beta_{k+1}}{2M\,\beta_k\,\gamma_k}\frac{(1-\beta_k)^2}{(1-\beta_{k+1})}, 
\label{eq:cond1_gamma}
\end{align}
\begin{align}
    1 - 2M\gamma_k^2 - M\beta_k^2\gamma_k^2 &\;\ge\; 0,
\label{eq:cond2_gamma}
\end{align}
\begin{align}
\frac{\gamma_{k+1}}{\gamma_k} \le
\frac{\beta_{k+1}}{2\,\beta_k(1-\beta_{k+1})}
\left[
\frac{4M\gamma_k^2 + \beta_k^4 - 2\beta_k^3 + 3\beta_k^2 - 2}{M(\beta_k^2+2)\gamma_k^2 - 1}
\right],
\label{eq:cond3_gamma}
\end{align}
for $M := 2L^2(1+\theta)$ and $\theta \ge 0$. With $\widetilde c_{k+1}\ge0$ for $\theta=2$, these give:
$$V_k-V_{k+1}\ge \widetilde c_{k+1}L^2\|x_{k+1}-y_k\|^2.$$

Conditions of \cref{eq:cond1_gamma}, \cref{eq:cond2_gamma}, \cref{eq:cond3_gamma} are satisfied, with constant step size $\gamma_k = \gamma_* \in \Bigl(0, \tfrac{1}{3L}\sqrt{\tfrac{5}{26}}\Bigr)$, and the choice of $\beta_k = \frac{2}{k+6}$.

\end{lemma}

\begin{theorem}\label{thm:main_gamma} {\normalfont [Proof in \cref{appx:main_gamma}.]}
Suppose $G$ is monotone and $L$-Lipschitz, and let $x^\star$ satisfy $G(x^\star)=0$. Consider the update of \eqref{eq:goma_square}. If the step size $\gamma$ and anchoring coefficient $\beta_k$ satisfy the conditions of \cref{lemma:potential_gamma}, then the potential function of \cref{eq:potential} is decreasing. This implies the bound
\begin{equation}\label{eq:simple_bound_gamma}
\|G(x_k)\|^2 
\;\le\; \frac{16/\gamma_*^2 + 32L^2}{(k+4)^2}\,\|x_0 - x^\star\|^2 .
\end{equation}
Moreover, the constant $16/\gamma_*^2+32L^2$ is decreasing in $\gamma_*$, so it is smallest at the largest admissible step; as $\gamma_*\to\tfrac{1}{3L}\sqrt{\tfrac{5}{26}}$ it gives
\begin{equation}\label{eq:edge_bound_2}
\|G(x_k)\|^2 
\;\le\; \frac{780.8\,L^2}{(k+4)^2}\,\|x_0 - x^\star\|^2 .
\end{equation}
\end{theorem}

\textbf{Summary.} In the deterministic monotone Lipschitz setting, the optimal $\mathcal{O}(1/k^2)$ last-iterate rate is already attained by several accelerated methods, including EAG \citep{pmlr-v139-yoon21d}, FEG \citep{lee2021fast}, and anchored Popov \citep{TranDinh2021HalpernTypeAA}. GOMA matches this optimal rate under both schedules (fixed $\eta_*$ and fixed $\gamma_*$). Beyond matching the rate, our analysis yields a pseudo fixed-step size scheme: hyperparameter tuning reduces to adjusting only $\eta_*$ or $\gamma_*$, with the two-time-scale structure maintained via the $(1-\beta_k)$ factor. This is simpler than the changing-step size argument required by anchored Popov \citep{TranDinh2021HalpernTypeAA} for a similar method. The main novelty of our work lies in the stochastic setting (Section~\ref{sec:stochastic}).

\section{Variant of GOMA for Stochastic Settings} \label{sec:stochastic}

\subsection{Background}

While first-order methods such as extragradient and optimistic gradient enjoy strong guarantees for deterministic variational inequalities, their behavior can fundamentally change in stochastic settings. In particular, under unbiased stochastic oracles, both may fail to converge, even for simple monotone problems \citep{NEURIPS2020_ba9a56ce}. 

A line of research \citep{doi:10.1137/070704277, 10.1214/10-SSY011, gorbunov2022clipped,beznosikov2023smooth, sadiev2023high, gorbunov2024high} establishes convergence guarantees for \textit{ergodic} (time-averaged) iterates of stochastic VI algorithms. However, ergodic convergence does not imply convergence of the actual iterates, which may cycle indefinitely despitefinite regret \citep{pmlr-v125-bailey20a}. 

This limitation has motivated approaches that modify the algorithmic dynamics to recover \textit{last-iterate} convergence in stochastic settings. \citet{NEURIPS2020_ba9a56ce} showed that introducing a two-time-scale scheme by using a larger step size for the extrapolation step restores almost sure \textit{last-iterate} convergence of stochastic extragradient for affine problems.

Beyond the affine case, last-iterate guarantees for general monotone Lipschitz VIs rely on one of two mechanisms, each with its own cost. The first is growing batch sizes \citep{lee2021fast}, which suppress the noise by drawing more samples each iteration. As a result, the per-iteration cost grows without bound and breaks the one-sample-per-round online model; in finite-sum problems it eventually reverts to full-batch gradients. By contrast, variance reduction \citep{pmlr-v178-alacaoglu22a, cai2022stochastic,chen2024near} keeps the batch fixed but recomputes periodic large-batch snapshots and stores a reference gradient within multi-phase schedules. Moreover, its empirical advantage over plain stochastic methods is reportedly limited for deep networks \citep{ NEURIPS2019_84d2004b}. We therefore ask whether last-iterate convergence is attainable with a single stochastic sample per iteration and constant, non-vanishing noise.


\subsection{Method}
We consider a simplified variant of \eqref{eq:goma} in which $\gamma_k = 0$,
leading to the following single–query update:
\begin{equation}\tag{$\diamond$}\label{eq:goma_gamma0}
\begin{aligned}
y_k &= \beta_k x_0 + (1-\beta_k)x_k,\\
x_{k+1} &= y_k - \eta_k G(y_k).
\end{aligned}
\end{equation}
This method evaluates the operator at an \emph{anchored interpolation}
between the current iterate and the initial point, and then applies a
single forward step at this interpolated point. 

Unlike optimistic or extragradient-type methods, this variation of GOMA does not reuse past gradients, which removes a key source of instability under stochastic noise. The update can be interpreted as an extreme form of time-scale separation, where extrapolation is replaced by anchoring to a fixed reference point, yielding a stable single-query method in regimes where classical approaches fail.
In the deterministic regime this simplification is slower: \eqref{eq:goma_gamma0} attains only an $\mathcal{O}(1/k)$ last-iterate rate on $\|G(x_k)\|^2$ (Theorem~\ref{thm:goma_gamma0_det}), rather than the accelerated $\mathcal{O}(1/k^2)$ of the general GOMA (Section~\ref{sec:deterministic}). We nonetheless adopt it because this structure is what enables stochastic last-iterate convergence.

Equation \eqref{eq:goma_gamma0} can also be written in the equivalent form:
\begin{align}
x_{k+1}
\!=\! x_k \!+\! \beta_k(x_0 \!-\! x_k)
\!-\! \eta_k G\!\left(x_k\! + \!\beta_k(x_0 \!-\! x_k)\right).
\label{eq:goma_momentum}
\end{align}
This formulation closely resembles Nesterov’s momentum method
\cite{nesterov1983method}:
\begin{align}\label{eq:nesterov}
x_{k+1}
\!=\! x_k\! +\! \beta_k (x_k\! -\! x_{k-1})\!
-\! \eta_k G\!\left(x_k\! +\! \beta_k (x_k\! -\! x_{k-1})\right).
\end{align}
The key difference is that \eqref{eq:goma_momentum} anchors the extrapolation
to the fixed point $x_0$, whereas Nesterov’s method \eqref{eq:nesterov}
anchors it to the previous iterate $x_{k-1}$. Moreover, the anchoring
directions are opposite: replacing $x_0$ by $x_{k-1}$ in
\eqref{eq:goma_gamma0} yields an update that matches Nesterov’s momentum
only when the momentum coefficient is negative.

An earlier work \cite{gidel2019negative} showed the effectiveness of
Polyak’s heavy-ball method \cite{polyak1964some} with negative momentum for
game dynamics. We will further compare this approach with GOMA empirically,
demonstrating the effectiveness of GOMA in stochastic settings.

We now turn to the convergence analysis of the stochastic variant
\eqref{eq:goma_gamma0}. In what follows, we state the assumptions under
which last-iterate convergence can be established, and then present the
corresponding rate guarantees.

\begin{assumption}
\label{ass:bounded_v}
We assume that $\widehat{G}(x,\xi)$ is an unbiased stochastic oracle for $G(x)$, i.e.,
\[
\mathbb{E}_{\xi}[\widehat{G}(x,\xi)\mid x] = G(x),
\]
and that, for some $\sigma \ge 0, \kappa >0$, the noise satisfies the second moment conditional bound
\[
\mathbb{E}_{\xi}\!\left[\|\widehat{G}(x,\xi)\|^2 \mid x\right]
\;\le\; \sigma^2 + \kappa \|G(x)\|^2.
\]
\end{assumption}
This assumption is similar to~\citep[Assump.~2]{NEURIPS2019_4625d8e3} and holds under mild conditions. 
We take $\kappa \geq 1$, without loss of generality.\footnote{Indeed, by unbiasedness and the decomposition $\mathbb{E}_\xi[\|\widehat G(x,\xi)\|^2 \mid x] = \|G(x)\|^2 + \mathbb{E}_\xi[\|\widehat G(x,\xi)-G(x)\|^2 \mid x]$, Assumption~\ref{ass:bounded_v} implies that $0\le \mathbb{E}_{\xi}\left[\|\widehat{G}(x,\xi) - G(x)\|^2 \mid x\right] \le \sigma^2 + (\kappa-1) \|G(x)\|^2$. If $\kappa < 1$, it follows that $\|G(x)\|^2 \leq \frac{\sigma^2}{1-\kappa}$ for all $x \in \R^d$, i.e., $G(x)$ is bounded. Furthermore, if Assumption~\ref{ass:bounded_v} holds for some pair of constants $(\sigma,\kappa_0)$, then it also holds for every $\kappa \ge \kappa_0$ with the same $\sigma$. Hence, any admissible $\kappa_0<1$ can be replaced by $\kappa=1$.} Crucially, our analysis covers any $\kappa\ge1$, i.e. state-dependent noise whose second moment grows with $\|G\|^2$; prior single-call stochastic VI guarantees (E-Halpern, RAIN++) and FEG require bounded variance ($\kappa=1$).

Under Assumption~\ref{ass:bounded_v}, we analyze~\eqref{eq:goma_gamma0} with a stochastic oracle $\widehat G$ in place of $G$:
\begin{equation}
     \label{eq:sto_updates}
     \begin{aligned}
y_k=\beta_k x_0+(1-\beta_k)x_k,\\ 
x_{k+1}=y_k- \eta_k\widehat G(y_k,\xi_{k}). \end{aligned}
\end{equation}
\paragraph{Proof strategy.}
Our analysis separates the deterministic and stochastic components of
the dynamics. We compare the noisy iterates to a \emph{deterministic
reference trajectory}, obtained by running the same
method~\eqref{eq:goma_gamma0} with the exact operator and the same
schedules:
\begin{equation}\label{eq:reference_traj}
\begin{aligned}
\bar x_0 &= x_0, \qquad
\bar y_k = \beta_k x_0 + (1-\beta_k)\,\bar x_k,\\
\bar x_{k+1} &= \bar y_k - \eta_k\, G(\bar y_k).
\end{aligned}
\end{equation}
By $L$-Lipschitzness,
$\|G(x_N)\|^2 \le 2\|G(\bar x_N)\|^2 + 2L^2\|x_N-\bar x_N\|^2$, so it
suffices to bound the residual along the reference trajectory and the
mean-square deviation of the stochastic iterates from it. The first
ingredient is purely deterministic: with the more conservative,
$\kappa$-dependent step size required by the stochastic setting, the
noiseless method retains an $\mathcal{O}(1/\sqrt{k})$ last-iterate guarantee,
both at the iterates $\bar x_k$ and at the query points $\bar y_k$.
Its proof is a potential-based analysis along the reference trajectory,
analogous to the deterministic case.

\begin{table*}[!h]
\centering
\caption{Last-iterate convergence guarantees on monotone Lipschitz VIs. Rates are last-iterate, reported on $\mathbb{E}\|G(x_k)\|$, i.e. the square root of the rates in Theorem~\ref{thm:last_iter_stoch_rho0}, unless explicitly marked ``(gap)'', in which case they are on $\mathbb{E}[\mathrm{GAP}]$ (not directly comparable to $\mathbb{E}\|G\|$ rates); ``(affine)'' marks a guarantee that holds only for affine operators. ``Unbounded domain'' indicates the method's analysis applies on $\mathbb{R}^d$; ``Single-call'' = one operator evaluation per iteration. \textbf{GOMA is the only method that is single-call, free of variance reduction and growing batches, and convergent under unbounded noise.}}
\label{tab:comparison}
\small
\resizebox{\textwidth}{!}{%
\begin{tabular}{l c c c c c c c}
\toprule
Method & \makecell{Unbounded \\ domain} & Single-call & \makecell{Deterministic  \\ rate} & \makecell{Stochastic \\ rate} & \makecell{No variance \\ reduction} & \makecell{No growing \\ batch} & \makecell{Unbounded \\ noise ($\kappa>1$)} \\
\midrule
DSEG \citep{NEURIPS2020_ba9a56ce}           & \checkmark & \texttimes & $\mathcal{O}  (k^{-1/2})$   & $\mathcal{O}(k^{-1/2})$ (affine)                  & \checkmark & \checkmark & \texttimes \\
FEG \citep{lee2021fast}                       & \checkmark & \texttimes & $\mathcal{O}  (k^{-1})$         & $\mathcal{O}(k^{-1})$                  & \checkmark & \texttimes & \texttimes \\
E-Halpern \citep{cai2022stochastic}           & \checkmark & \checkmark & $\mathcal{O}(k^{-1})$         & $\mathcal{O}(k^{-1/3})$                & \texttimes & \checkmark & \texttimes \\
RAIN$^{++}$ \citep{chen2024near}         & \checkmark & \texttimes & $\mathcal{O}(k^{-1})$         & $\tilde{\mathcal{O}}(k^{-1/2})$         & \texttimes & \checkmark & \texttimes \\
GABP \citep{abe2025boosting}                 & \texttimes & \checkmark & $\tilde{\mathcal{O}}(k^{-1})$ (gap) & $\tilde{\mathcal{O}}(k^{-1/7})$ (gap) & \checkmark & \checkmark & \texttimes \\
\midrule
\textbf{GOMA simplified} ($\gamma_k=0$)
  & \checkmark & \checkmark
  & \makecell[c]{$\mathcal{O}(k^{-1/2})$\\ {\scriptsize Thm~\ref{thm:goma_gamma0_det}}}
  & \makecell[c]{$\mathcal{O}(k^{-1/4})$\\ {\scriptsize Thm~\ref{thm:last_iter_stoch_rho0}}}
  & \checkmark & \checkmark & \checkmark \\
\bottomrule
\end{tabular}%
}
\end{table*}

\begin{lemma}[Deterministic reference bound]\label{lem:det-reference}
{\normalfont [Proof in \cref{appx:lem_det_reference}.]}
Let $G$ be monotone and $L$-Lipschitz with $G(x^\star)=0$, and let
$(\bar x_k,\bar y_k)$ be given by~\eqref{eq:reference_traj} with
$\beta_k=\frac{1}{k+2}$, $\eta_k=\frac{1}{L\sqrt{\kappa}(k+2)^{3/4}}$,
and $\kappa\ge1$. Then for all $N\ge1$ and all $k\ge0$,
\begin{align}
\|G(\bar x_N)\|^2
& \le 
\frac{33\,L^2\kappa\,\|x_0-x^\star\|^2}{\sqrt{N+1}},
\label{eq:lemA_xbound}\\
\|G(\bar y_k)\|^2
&\le \frac{94\,L^2\kappa\,\|x_0-x^\star\|^2}{\sqrt{k+1}}.
\label{eq:lemA_ybound}
\end{align}
\end{lemma}
 
The bound~\eqref{eq:lemA_ybound} at the query points is the
technically important addition: under
Assumption~\ref{ass:bounded_v} the second moment of the oracle grows
with $\|G(y_k)\|^2$, so controlling the noise injected at step $k$
requires a residual bound along the \emph{entire} trajectory, not only
at the final iterate. The second ingredient shows that the noisy
iterates track the reference trajectory.
 
\begin{lemma}[Stochastic stability]\label{lem:stoch-stability}
{\normalfont [Proof in \cref{appx:lem_stoch_stability}.]}
In the setting of Lemma~\ref{lem:det-reference}, let $\widehat G$
satisfy Assumption~\ref{ass:bounded_v}, let $(x_k)$ be given
by~\eqref{eq:sto_updates}, and set $e_k:=x_k-\bar x_k$. Then for all $N\ge0$,
\begin{equation}\label{eq:lemB_precise}
\mathbb{E}\,\|e_N\|^2
\;\le\; 
\frac{1}{\sqrt{N+1}}
\left(\frac{4\,\sigma^2}{L^2\kappa}
+ 752\,\kappa\,\|x_0-x^\star\|^2\right).
\end{equation}
\end{lemma}
 
The two terms in~\eqref{eq:lemB_precise} mirror the two noise sources
in Assumption~\ref{ass:bounded_v}: the additive variance $\sigma^2$
and the state-dependent part $\kappa\|G\|^2$, the latter controlled
through~\eqref{eq:lemA_ybound}. The key mechanism is that anchoring
contracts the deviation: the interpolation toward $x_0$ multiplies the
error by $(1-\beta_k)$ at every step, which yields a contraction rate
of $1-\Theta(1/k)$ in the error recursion. This contraction is strong
enough to keep the accumulated noise at $\mathcal{O}(1/\sqrt{N})$ with a single
sample per iteration, without variance reduction. Combining the two lemmas
through the Lipschitz decomposition above yields our main stochastic
guarantee.

\begin{theorem}[Last-iterate bound for stochastic GOMA] \label{thm:last_iter_stoch_rho0} {\normalfont [Proof in \cref{appx:last_iter_stoch_rho0}.]} Let \(G:\mathbb{R}^d\to\mathbb{R}^d\) be monotone and \(L\)-Lipschitz and $\widehat G(x,\xi)$ be a stochastic oracle following Assumption \ref{ass:bounded_v}. Then for the updates described in~\eqref{eq:sto_updates} with $\beta_k=\frac{1}{k+2}$ and $\eta_k=\frac{1}{L\sqrt{\kappa} (k+2)^{3/4}}$, we have for all $N\ge 0$,
\begin{align*}
\mathbb{E}\|G(x_N)\|^2 \;\le\; \frac{ 1570 L^2 \kappa \|x_0-x^\star\|^2}{\sqrt{N+1}} + \frac{8\,\sigma^2 }{\kappa\sqrt{N+1}}.
\end{align*}
\end{theorem}

Theorem~\ref{thm:last_iter_stoch_rho0} establishes, to the best of our knowledge, the first last-iterate convergence guarantee in the squared operator norm $\mathbb{E}\|G(x_N)\|^2$ for unconstrained stochastic monotone Lipschitz VIs, \textbf{without variance reduction or growing batch sizes} and the guarantee holds for every $\kappa\ge1$, \textbf{covering state-dependent noise} whose variance can grow unboundedly with $\|G(x_N)\|^2$. Concretely, GOMA attains $\mathbb{E}\|G(x_N)\|^2 \le \varepsilon$ in $N = \mathcal{O}(1/\varepsilon^2)$ iterations. 

We now place our stochastic guarantees in context by comparing them to existing last-iterate results.

\textbf{Comparison with existing methods.} FEG~\citep{lee2021fast} performs two operator evaluations per iteration, and its stochastic guarantee requires the per-iteration variance to decay as $\sigma_k = \mathcal{O}(1/k)$; under constant noise the error term accumulates as $\mathcal{O}(k)$, so the bound no longer vanishes and the last-iterate guarantee is lost unless growing minibatches enforce the decay, whereas GOMA converges with a constant batch size and non-vanishing noise. 

E-Halpern~\citep{cai2022stochastic} builds on anchored Popov with recursive variance reduction (PAGE) to obtain a single-call algorithm with SFO complexity $\mathcal{O}(1/\varepsilon^3)$ on $\mathbb{E}\|G\|$, at the additional cost of assuming Lipschitz continuity of the stochastic oracle in expectation; GOMA removes both the variance reduction and this oracle assumption, at a slower $\mathcal{O}(1/\varepsilon^4)$ complexity. 

RAIN/RAIN++~\citep{chen2024near} combine anchoring with recursive variance reduction to obtain near-optimal stochastic first-order oracle (SFO) complexity for smooth convex–concave minimax problems, matching the lower bound they derive up to logarithmic factors, but through a multi-phase scheme with restarting schedules whose returned point is an iterate sampled uniformly along the trajectory rather than the last one, while GOMA uses a single-call, single-phase update and provably returns the last iterate. 

Most recently, \citet{alacaoglu2026solving} remove the bounded-variance assumption entirely, allowing the variance to grow with $\|z\|^2$, and reach an $\tilde{\mathcal{O}}(1/\varepsilon^4)$ residual complexity using multilevel Monte Carlo or STORM variance reduction together with growing batches, again reporting a randomly selected iterate; GOMA attains its last-iterate guarantee without variance reduction or growing batches.

A separate line of work targets the constrained regime with the gap function as progress measure. \citet{abe2025boosting} propose GABP, a single-call payoff-perturbed algorithm with periodic anchor restarts, and prove $\mathbb{E}[\mathrm{GAP}(\pi^{T+1})] = \tilde{\mathcal{O}}(1/T^{1/7})$ under bounded-variance noise on a compact domain $X$, with a uniform operator bound valid only on compact $X$. This setting is not directly comparable to our unconstrained $\mathbb{R}^d$, squared-operator-norm regime (Section~\ref{sec:prelim} explains why these measures are not interchangeable); within these complementary regimes GOMA is faster, $\mathcal{O}(1/T^{1/4})$ on $\mathbb{E}\|G(x_T)\|$ versus their $\tilde{\mathcal{O}}(1/T^{1/7})$ on $\mathbb{E}[\mathrm{GAP}]$.

\textbf{Summary.} As Table~\ref{tab:comparison} shows, GOMA occupies a distinctive point in this design space: it is the only method that guarantees \emph{last-iterate} convergence in the squared operator norm using a single operator evaluation per iteration, a constant batch size, and no variance reduction or restarts, while tolerating state-dependent noise whose variance grows with $\|G\|^2$. The other methods achieve faster rates, but only through variance reduction or growing batches that reduce or eliminate the noise terms complicating stochastic last-iterate analysis, and several return an averaged or randomly selected iterate rather than the last. Concretely, GOMA attains a last-iterate rate of $\mathcal{O}(1/\sqrt{k})$ on $\mathbb{E}\|G(x_k)\|^2$ (equivalently $\mathcal{O}(1/k^{1/4})$ on $\mathbb{E}\|G(x_k)\|$), i.e.\ an SFO complexity of $\mathcal{O}(1/\varepsilon^4)$. This does not match the optimal rate of $\tilde{\mathcal{O}}(1/k)$ on $\mathbb{E}\|G(x_k)\|^2$ ($\tilde{\mathcal{O}}(1/\varepsilon^2)$ SFO complexity), which \citet{chen2024near} establish as a lower bound and which variance-reduced, growing-batch methods attain; closing this gap without such mechanisms remains an open question. Within the class of methods that use neither variance reduction nor growing batches, GOMA provides the first and best stochastic last-iterate guarantee on monotone Lipschitz VIs.

\section{Experiments}

\subsection{Negative-Comonotone Quadratic Saddle Point (Deterministic)}\label{sec:quadratic}

\textbf{Setup.} We performed a toy experiment on a simple quadratic function also used in \citep{lee2021fast},
\begin{equation}
f(x,y) = -\tfrac{1}{6}x^2 + \tfrac{2\sqrt{2}}{3}xy + \tfrac{1}{6}y^2.
\end{equation}
This instance is $\rho$-comonotone with $\rho = -1/3 < 0$ (i.e., \emph{negative comonotone}), which lies outside the scope of our theory (our analysis requires monotonicity, $\rho \geq 0$). We include it for direct comparison with prior work \citep{lee2021fast}, and provide an additional experiment on a \emph{monotone} instance covered by our theory, which is in Appendix~\ref{appx:exp_monotone_qp}.

\textbf{Methods.}
We compare several first-order methods on this problem, including
EG, DSEG, EAG-C, EAG-V, Nesterov, FEG, anchored Popov, and our proposed GOMA.
For GOMA, we use a constant step size $\eta = 0.2$ and 
$\gamma_k = 0.8(1 - \beta_k)$ with $\beta_k = \frac{2}{k+6}$.
All methods are evaluated by plotting the squared operator norm
$\|F(z_k)\|^2$ against the number of gradient calls.
\begin{figure}[t]
    \centering
    \includegraphics[width=0.95\linewidth]{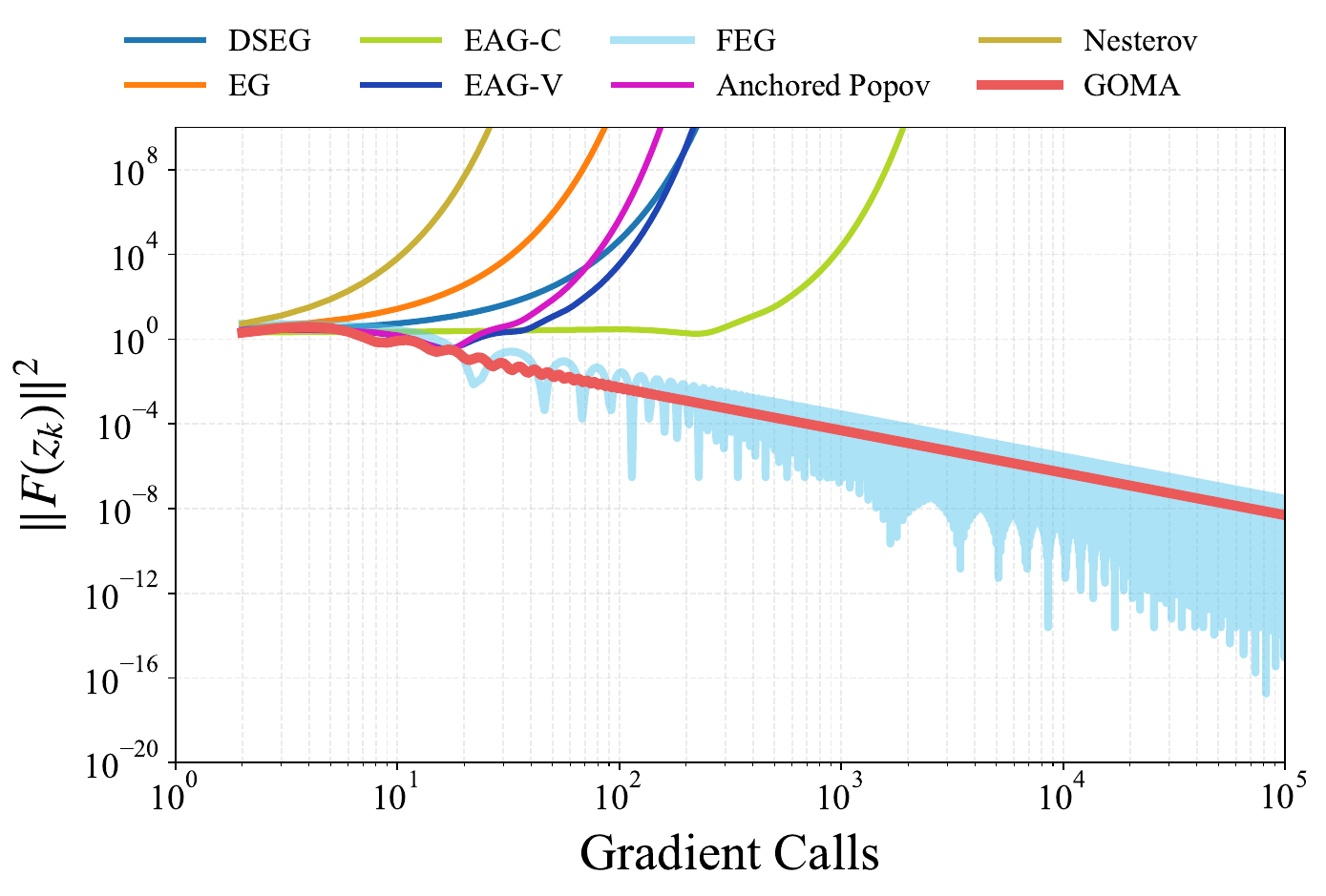}
    \caption{Quadratic experiment in the deterministic setting ~\S\ref{sec:quadratic}. Only GOMA and FEG converge. \textbf{GOMA converges without oscillations}, unlike FEG’s dynamics.}
    \label{fig:comparison}
\end{figure}

\captionsetup{font=small,labelfont=bf,justification=raggedright,singlelinecheck=false}

\begin{figure*}[t]                             
  \centering
  \begin{subfigure}[t]{0.42\textwidth}
    \centering
    \includegraphics[width=0.95\linewidth,trim=2pt 4pt 2pt 2pt,clip]{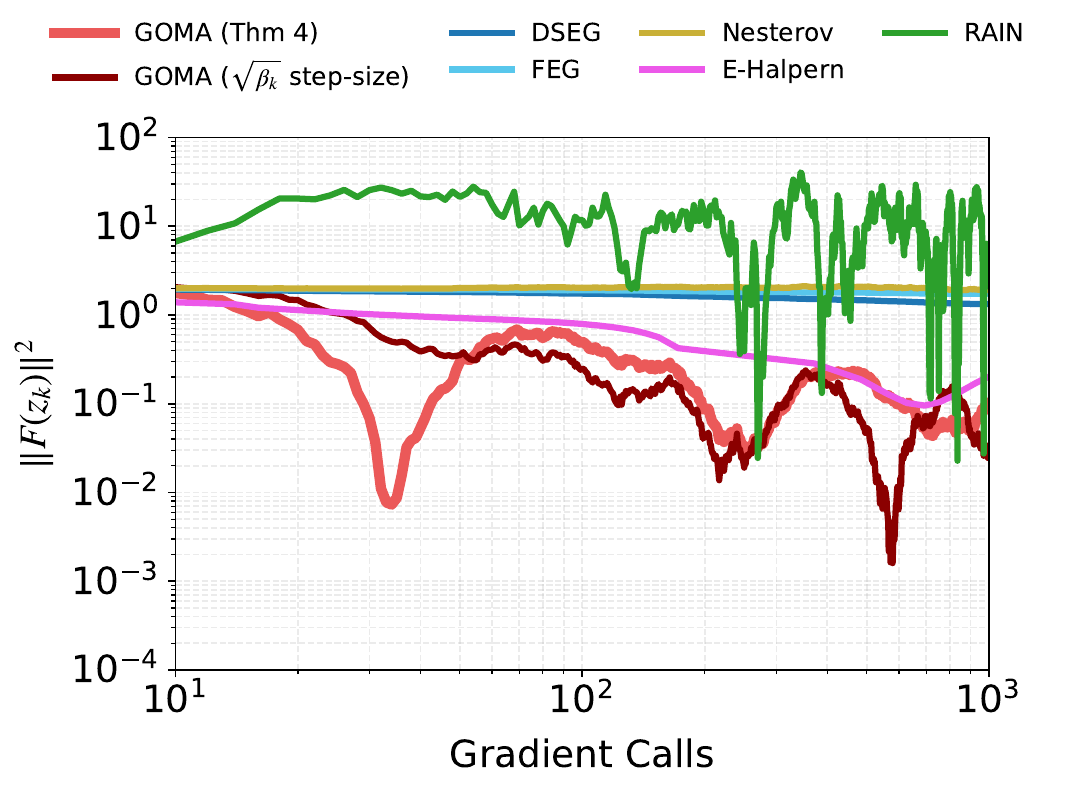}
    \label{fig:goma}
  \end{subfigure}\hfill
  \begin{subfigure}[t]{0.42\textwidth}
    \centering
    \includegraphics[width=0.95\linewidth,trim=2pt 4pt 2pt 2pt,clip]{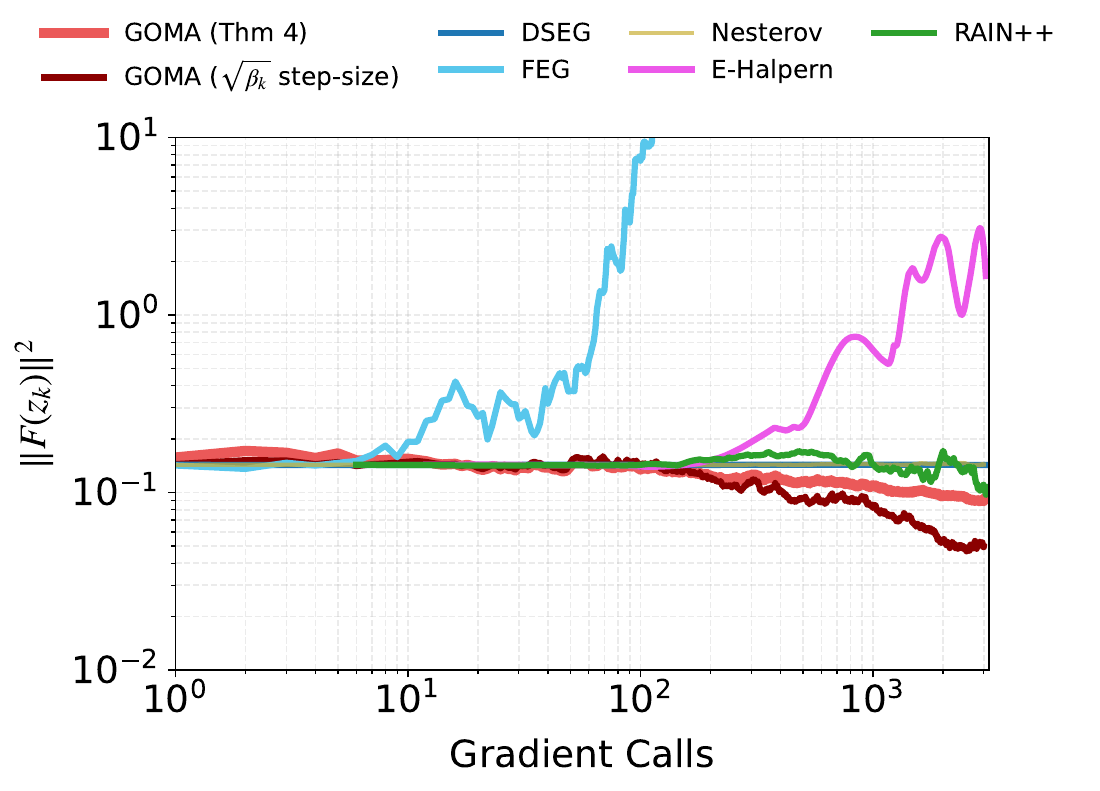}
    \label{fig:goma3_sto}
  \end{subfigure}
  \caption{Stochastic experiments ~\S\ref{sec:exp_stochastic}. Norm of the gradient operator vs. the number of gradient calls. 
\textbf{Left:} Stochastic bilinear example with additive noise ($\sigma = 0.5$), ~\S\ref{exp_stoch_additive}. 
\textbf{Right:} Finite-sum saddle-point problem with multiplicative noise, ~\S\ref{sec:exp_stoch_multiplicative}. 
\textbf{Stochastic \textsc{GOMA} consistently outperforms the baselines in both tasks}, despite using a milder hyperparameter search and simpler setup.}
  \label{fig:exp_stochastic}
\end{figure*}

\textbf{Hyperparameter selection.}
All baseline methods are run under the identical experimental setting of
Figure~2 in \cite{lee2021fast}, including the same initialization, step size rules,
and algorithmic parameters.
This allows for a direct and fair comparison with the original results, to which
we additionally include our proposed GOMA. 

\textbf{Results.}
See Figure ~\ref{fig:comparison}. \textsc{GOMA} and \textsc{FEG} converge with an accelerated rate,
whereas EG, \textsc{DSEG}, \textsc{EAG-C}, \textsc{EAG-V}, Anchored Popov and Nesterov diverge.
Moreover, on this instance (under our tuning), \textsc{GOMA} yields uniformly smaller residuals than \textsc{FEG} by an approximately constant factor, with near-parallel log–log curves indicating the same asymptotic rate but a better constant.

\subsection{Stochastic Games}\label{sec:exp_stochastic}

We consider two stochastic settings. 
The first is a low-dimensional toy problem ($d=2$) that satisfies the additive Gaussian noise assumption ($\kappa = 1$). The second is a finite-sum saddle-point problem ($d=10$) with state-dependent multiplicative noise ($\kappa > 1$). Theorem~\ref{thm:last_iter_stoch_rho0} provides guarantees in both regimes, while the stochastic theory of FEG, E-Halpern, and RAIN++ applies only to the first.

\textbf{Methods.}
We compare GOMA \eqref{eq:goma_gamma0} against DSEG, FEG, E-Halpern (with PAGE variance reduction), RAIN++, and
Nesterov's accelerated method (with negative momentum, following \cite{gidel2019negative}). GOMA is run with a \textbf{single stochastic sample per iteration}, matching the constant-batch setting of Theorem~\ref{thm:last_iter_stoch_rho0}; we use no variance reduction and no growing batch size. Several of the baselines (E-Halpern, RAIN++) rely on variance reduction by design, and we use their authors' recommended implementations. 

\subsubsection{Case I: Bounded Variance ($\kappa = 1$)}\label{exp_stoch_additive}

\textbf{Setup.}
We consider the stochastic bilinear game
\begin{equation}
    f(x, y) = Lxy, \quad L = 1,
\end{equation}
with saddle operator $F(z) = F(x,y) = (Ly, -Lx)$ and solution $z^\star = (0,0)$. The stochastic oracle is given by $\widehat{F}(z, \xi) = F(z) + \xi$, where $\xi \sim \mathcal{N}(0, \sigma^2 I)$ with $\sigma = 0.5$. Since the noise is state-independent, this setting satisfies Assumption~\ref{ass:bounded_v} with $\kappa = 1$. We initialize at $z_0 = (1, 1)$ and run all methods for $10^3$ gradient calls. The details about the hyperparameter selection are deferred to Appendix~\ref{appx:exp_stoch_detail_kappa_1}.



\textbf{Results.}
Figure~\ref{fig:exp_stochastic} (left) compares the convergence of the squared operator norm under additive noise. \textbf{GOMA} achieves the fastest convergence, reaching a residual nearly an order of magnitude smaller than all baselines within $10^3$ gradient calls. E-Halpern exhibits steady progress due to its variance-reduction mechanism; however, it converges more slowly and plateaus around $10^{-1}$.

The two-time-scale extragradient-type methods, DSEG and FEG, show little to no convergence, remaining close to their initial values throughout the experiment. RAIN displays highly unstable behavior with large oscillations and fails to converge. 
These results highlight the effectiveness of combining anchoring with single-call stochastic gradients: \textbf{GOMA achieves the best convergence without relying on variance reduction or multi-point oracle evaluations}.

\subsubsection{Case II: State-Dependent Variance ($\kappa > 1$)}\label{sec:exp_stoch_multiplicative}

\textbf{Setup.}
We consider the finite-sum saddle problem
\[
\min_{\theta\in\R^{d}}\ \max_{\varphi\in\R^{d}}\;
\frac{1}{n}\sum_{i=1}^{n}\Big(\theta^{\top}b_i+\theta^{\top}A_i\varphi+c_i^{\top}\varphi\Big),
\]
with saddle operator \(F(\theta,\varphi)=[\bar b+\bar A\varphi;\ -(\bar A^{\top}\theta+\bar c)]\),
where \(\bar A=\tfrac1n\sum_i A_i\) and \(\bar b,\bar c\) are the sample means.
We set \(n=d=10\) and \(A_i=\mathrm{diag}(0,\ldots,\lambda_i,\ldots,0)\) with \(\lambda_i\) equally spaced in \([\tau,1]\), \(\tau=0.1\).
The SFO returns \(F_i(\theta,\varphi)=[b_i+A_i\varphi;\ -(A_i^{\top}\theta+c_i)]\), sampling \(i\) uniformly.

\textbf{Hyperparameter selection.}
For GOMA, we use $\beta_k = 1/(k+2)$ and $\eta_k = c\sqrt{\beta_k}$ with $c$ selected via grid search. Full details are provided in Appendix~\ref{app:hyperparams}.

\textbf{Results.}
Figure~\ref{fig:exp_stochastic} (right) illustrates the convergence behavior of the squared operator norm under multiplicative noise. This setting is covered by Theorem~\ref{thm:last_iter_stoch_rho0} but falls outside the bounded-variance assumptions of FEG, E-Halpern, and RAIN++. Both RAIN++ and \textsc{GOMA} exhibit convergence empirically. In contrast, DSEG stagnates at a high plateau (around $10^{-1}$) and fails to make further progress. As shown in Appendix~\ref{appx:exp_stochastic}, its theory predicts arbitrarily slow convergence in higher dimensions, which is consistent with this behavior.

Overall, these results show that \textbf{\textsc{GOMA} converges in this multiplicative noise regime as guaranteed by Theorem~\ref{thm:last_iter_stoch_rho0}, while the bounded-variance baselines (FEG, E-Halpern) exit their theoretical regime and diverge}.

\section{Discussion}

Our results show that anchoring, when combined with generalized optimistic dynamics, offers a principled approach to addressing three central challenges in variational inequality algorithms: per-iteration efficiency, robustness to stochasticity, and last-iterate acceleration.

In particular, GOMA attains the optimal $\mathcal{O}(1/k^2)$ deterministic last-iterate rate using a single gradient evaluation per iteration, and a $\mathcal{O}(1/\sqrt{k})$ stochastic last-iterate rate without the variance reduction or growing batches that prior last-iterate guarantees depend on. This shows that strong last-iterate guarantees are compatible with the one-sample online model, and our experiments confirm that GOMA is the most robust method under unbounded-variance noise.

The most important open direction for the stochastic theory is to close the gap between the $\mathcal{O}(1/\varepsilon^4)$ SFO complexity of GOMA and the $\tilde{\mathcal{O}}(1/\varepsilon^2)$ SFO lower bound established by \citet{chen2024near}, without resorting to variance reduction or growing batches. Last-iterate convergence of GOMA in the constrained setting, where the convergence measure changes from $\|G\|^2$ to the gap function, remains open. Beyond the monotone setting, extending the analysis to broader operator classes such as negative comonotone operators is another natural direction. Finally, applying these methods at scale in reinforcement learning and adversarial training, where stochasticity, stability, and gradient efficiency are central concerns, is an exciting practical direction.

\section*{Impact Statement}
This paper presents work aimed at advancing the field of Variational Inequality Problems. Our work being focused on the theoretical aspect, we do not foresee any direct societal impact. Regarding indirect impact, while a common positive, foreseeable impact of such research aiming to find better optimisation algorithms is the more efficient use of computing resources, it is important to be mindful that, across history, cost-lowering technological improvements have nevertheless often led to an increase in consumption due to the Jevons paradox. 

\section*{Acknowledgements}

This research was partially supported by the Canada CIFAR AI Chair program (Mila), Simon Lacoste-Julien is a CIFAR Associate Fellow in the Learning in Machines \& Brains program.

We would like to thank Zichu Liu, Juan David Guerra and Mehran Shakerinava for their feedbacks on the initial draft of this paper.

We also acknowledge the use of AI assistants during this work.
ChatGPT (OpenAI) and Claude Code (Anthropic) helped identify errors in
intermediate steps while we were developing the proofs of our main
theorems. In particular, ChatGPT (OpenAI, Pro mode) found a sign error
that invalidated an earlier version of our stochastic analysis and
proposed the corrected proof strategy for
Theorem~\ref{thm:last_iter_stoch_rho0}, based on a deterministic
reference trajectory and a stochastic stability argument.


\bibliography{example_paper}
\bibliographystyle{icml2026}

\appendix
\crefalias{section}{appendix}
\crefalias{subsection}{appendix}

\onecolumn

\addcontentsline{toc}{section}{Appendix}
\part{Appendix} 
\parttoc

\section{Proof of \cref{sec:deterministic}}

\textbf{Generalized Optimistic Method with Anchoring (GOMA):}
\begin{equation}\tag{GOMA}
\begin{aligned}
y_{k} &= \beta_k x_0 + (1 - \beta_k)x_k - \gamma_k G (y_{k-1}) \\
x_{k+1} &= \beta_k x_0 + (1 - \beta_k)x_k - \eta_k G(y_{k}).
\end{aligned}
\end{equation}

Define the potential function:
\begin{equation}\repeq{eq:potential}
     V_k = a_k\|G(x_k)\|^2 + b_k\langle G(x_k),\,x_k - x_0\rangle + c_kL^2\|x_k - y_{k-1}\|^2,
\end{equation}
where $a_k = c_k = \frac{b_k\eta_k}{2\beta_k}$, and $b_{k+1} = \frac{b_k}{1 - \beta_k}$.

\subsection{Proof of \cref{lemma:potential_eta}}\label{appx:lemma1}

This proof is inspired by the proof of anchored Popov from \citet{TranDinh2021HalpernTypeAA}.

\begin{replemma}{lemma:potential_eta}
Let $G$ be monotone and $L$-Lipschitz, and consider the iterates
\begin{equation}\repeq{eq:goma_triangle}
\begin{aligned}
y_k &= \beta_kx_0 + (1-\beta_k)x_k - \eta_*(1-\beta_k) G(y_{k-1}), \\
x_{k+1} &= \beta_kx_0 + (1-\beta_k)x_k - \eta_* G(y_k).
\end{aligned}
\end{equation}

We prove that the potential function \eqref{eq:potential} is decreasing if the step-size $\eta_k$ satisfies the following conditions:
\begin{equation}\repeq{eq:cond1}
\eta_{k+1}\;\le\;\frac{\beta_{k+1}}{2M\,\eta_k\,\beta_k(1-\beta_k)}.
\end{equation}
\begin{equation}\repeq{eq:cond2}
    1 - 2M\eta_k^2(1 - \beta_k)^2 - M\eta_k^2\beta_k^2 \ge 0.
\end{equation}
\begin{equation}\repeq{eq:cond3}
\hspace{-1.5em}\scalebox{0.85}{$\displaystyle
\eta_{k+1}
\le
\frac{\beta_{k+1}}{2\,\beta_k(1-\beta_k)}
\Bigl[\frac{2(1 - \beta_k^2) -4M\eta_k^2(1 - \beta_k)^2 -\beta_k^2}
{1 - 2M\eta_k^2(1 - \beta_k)^2 - M\eta_k^2\beta_k^2}\eta_k
\Bigr].$}
\end{equation}

for $M := 2L^2(1+\theta)$ and $\theta \ge 0$. With $\widetilde c_{k+1}\ge0$ for $\theta=2$, these give:
$$V_k-V_{k+1}\ge \widetilde c_{k+1}L^2\|x_{k+1}-y_k\|^2.$$
\end{replemma}

\begin{proof}
First, from the update equations we obtain the three key difference identities:
\begin{align}
x_{k+1}-x_k &= \beta_k(x_0-x_k) - \eta_k G(y_k), \label{eq:h1}\\
x_{k+1}-x_k &= \frac{\beta_k}{1-\beta_k}(x_0 - x_{k+1}) - \frac{\eta_k}{1-\beta_k}G(y_k),\label{eq:h2}\\
x_{k+1}-y_k &= -\eta_k\bigl[G(y_k)-(1 - \beta_k)G(y_{k-1})\bigr]\label{eq:h3}.
\end{align}
Next, monotonicity of $G$ gives

$$
  \langle G(x_{k+1})-G(x_k),\,x_{k+1}-x_k \rangle \ge 0
$$

$$
  \langle G(x_{k+1}),\,x_{k+1}-x_k \rangle \ge \langle G(x_k),\, x_{k+1}-x_k\rangle
$$
Using \cref{eq:h1} and \cref{eq:h2}, we write:
$$
  \langle G(x_{k+1}),\, \frac{\beta_k}{1-\beta_k}(x_0 - x_{k+1}) - \frac{\eta_k}{1-\beta_k}G(y_k) \rangle \ge \langle G(x_k),\,\beta_k(x_0-x_k) - \eta_k G(y_k) \rangle
$$
Rearranging:
$$
  \frac{\beta_k}{1-\beta_k}\langle G(x_{k+1}),\, x_0 - x_{k+1} \rangle  \ge \beta_k\langle G(x_k),\,x_0-x_k \rangle - \eta_k\langle G(x_k),\, G(y_k) \rangle +  \frac{\eta_k}{1-\beta_k}\langle G(x_{k+1}),\, G(y_k)\rangle
$$
Multiplying this inequality by $\frac{b_k}{\beta_k}$ and taking $b_{k+1} = \frac{b_k}{1 - \beta_k}$:

\begin{align*}
    \underbrace{b_k\langle G(x_k),x_k-x_0\rangle - b_{k+1}\langle G(x_{k+1}),x_{k+1}-x_0\rangle}_{T[1]}
    \ge &
    \frac{b_k\,\eta_k}{\beta_k(1-\beta_k)}\langle G(x_{k+1}),G(y_k)\rangle
    - \frac{b_k\,\eta_k}{\beta_k}\langle G(x_k),G(y_k)\rangle.\\
    =& b_{k+1}\,\eta_k\langle G(x_{k+1}),G(y_k)\rangle
    + \frac{b_k\,\eta_k}{\beta_k}\langle G(x_{k+1}) - G(x_k),G(y_k)\rangle.
\end{align*}

Adding the $a_k$-terms and the $c_k$-terms gives

\begin{align}
    V_k - V_{k+1}
 \ge&
 a_k\|G(x_k)\|^2 - a_{k+1}\|G(x_{k+1})\|^2
 + T[1]
 + c_kL^2\|x_k-y_{k-1}\|^2 - c_{k+1}L^2\|x_{k+1}-y_k\|^2\nonumber\\
 \ge&  a_k\|G(x_k)\|^2 - a_{k+1}\|G(x_{k+1})\|^2 \nonumber\\
 &+ b_{k+1}\,\eta_k\langle G(x_{k+1}),G(y_k)\rangle
    + \frac{b_k\,\eta_k}{\beta_k}\langle G(x_{k+1}) - G(x_k),G(y_k)\rangle\nonumber\\
 &+ c_kL^2\|x_k-y_{k-1}\|^2 - c_{k+1}L^2\|x_{k+1}-y_k\|^2 \label{eq:inter-v_diff}
\end{align}

Next, we upper bound $\|G(y_k) - G(y_{k-1})\|^2$ as follow:
\begin{align}
\|G(y_k) - G(y_{k-1})\|^2 
&= \|G(y_k) - G(x_k) + G(x_k) - G(y_{k-1})\|^2 \notag \\
&\leq 2\|G(x_k) - G(y_k)\|^2 + 2\|G(x_k) - G(y_{k-1})\|^2 \notag\\
&\leq 2\|G(x_k)\|^2 - 4\langle G(x_k), G(y_k) \rangle + 2\|G(y_k)\|^2 + 2L^2\|x_k - y_{k-1}\|^2.\label{eq:Gyk-Gykp1}
\end{align}
Where we used the Lipschitz inequality between $x_k$ and $y_{k-1}$ in the last inequality.   
\bigskip

We consider the following and use Lipschitzness between $x_{k+1}$ and $y_k$ and \cref{eq:h3} to bound the left-hand side

\begin{align}
    \|G(x_{k+1}) - G(y_k)\|^2 + \theta L^2\|x_{k+1}-y_k\|^2 &\leq  (1+\theta)L^2\|x_{k+1}-y_k\|^2\\
    &\leq (1+\theta)L^2\eta_k^2\bigl\|G(y_k)-(1-\beta_k)\,G(y_{k-1})\bigr\|^2 \label{eq:h4}
\end{align}

We can upper bound the right-hand side. We first use the inequality $\|a + b\|^2 \leq 2\|a\|^2 + 2\|b\|^2$. Then we use the bound in the upper bound on $\|G(y_k) - G(y_{k-1})\|^2 $.

\begin{align}
    &(1+\theta)L^2\eta_k^2\bigl\|G(y_k)-(1-\beta_k)\,G(y_{k-1})\bigr\|^2  \nonumber\\
& \le 2L^2(1+\theta)(1 - \beta_k)^2\,\eta_k^2\,\bigl\|G(y_k)-\,G(y_{k-1})\bigr\|^2 + 2L^2(1+\theta)\,\eta_k^2\,\beta_k^2\|G(y_{k})\|^2\nonumber\\
& \overset{\eqref{eq:Gyk-Gykp1}}{\leq} 4L^2(1+\theta)\,\eta_k^2(1 - \beta_k)^2\,\bigl(\|G(x_k)\|^2 - 2\langle G(x_k), G(y_k) \rangle +  \|G(y_k)\|^2) \nonumber\\
    &  + 4L^4(1+\theta)\,\eta_k^2(1 - \beta_k)^2\,\|x_k - y_{k-1}\|^2 +  2L^2(1+\theta)\,\eta_k^2\,\beta_k^2\|G(y_{k})\|^2
\end{align}

Now we expand the quadratic in the left-hand side of \cref{eq:h4} and also substitute the right-hand side from above.

\begin{align*}
    &\|G(x_{k+1})\|^2 - 2\langle G(x_{k+1}), G(y_k) \rangle + \|G(y_k)\|^2+ \theta L^2\|x_{k+1}-y_k\|^2\\
    & \leq 4L^2(1+\theta)\,\eta_k^2(1 - \beta_k)^2\,\bigl(\|G(x_k)\|^2 - 2\langle G(x_k), G(y_k) \rangle +  \|G(y_k)\|^2) \nonumber\\
    &  + 4L^4(1+\theta)\,\eta_k^2(1 - \beta_k)^2\,\|x_k - y_{k-1}\|^2 +  2L^2(1+\theta)\,\eta_k^2\,\beta_k^2\|G(y_{k})\|^2
\end{align*}

Rearranging and setting $M := 2L^2(1+\theta)$, we get:

\begin{align*}
&\|G(x_{k+1})\|^2
-2\langle G(x_{k+1}),\,G(y_k)\rangle
-2M\eta_k^2(1 - \beta_k)^2\,\|G(x_k)\|^2
+4M\eta_k^2(1-\beta_k)^2\,\langle G(x_k),\,G(y_k)\rangle\\
&\quad+\;[1 - 2M\eta_k^2(1 - \beta_k)^2 - M\eta_k^2\beta_k^2]\,\|G(y_k)\|^2
+\theta L^2\,\|x_{k+1}-y_k\|^2\\
&\quad-\;2L^2M\eta_k^2(1-\beta_k)^2\,\|x_k-y_{k-1}\|^2
\;\le\;0.
\end{align*}

Combine terms to get $G(x_{k+1}) - G(x_k)$:
\begin{align*}
&\|G(x_{k+1})\|^2
-2(1 -2M\eta_k^2(1-\beta_k)^2)\langle G(x_{k+1}),\,G(y_k)\rangle
-2M\eta_k^2(1 - \beta_k)^2\,\|G(x_k)\|^2
\\&\quad-4M\eta_k^2(1-\beta_k)^2\,\langle G(x_{k+1}) - G(x_k),\,G(y_k)\rangle
+\;[1 - 2M\eta_k^2(1 - \beta_k)^2 - M\eta_k^2\beta_k^2]\,\|G(y_k)\|^2\\
&\quad+\theta L^2\,\|x_{k+1}-y_k\|^2
-\;2L^2M\eta_k^2(1-\beta_k)^2\,\|x_k-y_{k-1}\|^2
\;\le\;0.
\end{align*}

We multiply this equation by $\frac{a_k}{2M\eta_k^2(1 - \beta_k)^2}$ and add it  to right-hand side of \cref{eq:inter-v_diff} get:

\begin{align}
V_k - V_{k+1}
\;&\ge\;
\underbrace{\Bigl(\frac{a_k}{2M\eta_k^{2}(1-\beta_k)^2}-a_{k+1}\Bigr)}_{S_k^{11}}
        \,\|G(x_{k+1})\|^{2}                                      \nonumber\\
&\quad
+\,2\underbrace{\Bigl(-\frac{a_k (1 - 2M\eta_k^2(1 - \beta_k)^2)}{2M\eta_k^2(1 - \beta_k)^2} +  \frac{b_{k+1}\eta_k}{2} \Bigr)}_{S_k^{12}}
        \,\langle G(x_{k+1}),\,G(y_k)\rangle                      \nonumber\\
&\quad
+\,\underbrace{\Bigl(\frac{a_k}{2M\eta_k^{2}(1 - \beta_k)^2}(1 - 2M\eta_k^2(1 - \beta_k)^2 - M\eta_k^2\beta_k^2)\Bigr)}_{S_k^{22}}
        \,\|G(y_k)\|^{2}                                                                   \nonumber\\
&\quad
+\,\underbrace{\Bigl(-2a_k+\frac{b_k\eta_k}{\beta_k}\Bigr)}_{S_k^{23}}
        \,\langle G(x_{k+1}) - G(x_k),\,G(y_k)\rangle                          \nonumber\\[4pt]
&\quad
+\,L^{2}\underbrace{\Bigl(c_k-a_k\Bigr)}_{\widetilde c_k}
        \,\|x_k-y_{k-1}\|^{2}                                     \nonumber\\
&\quad
+\,L^{2}\underbrace{\Bigl(\frac{a_k\theta}{2M\eta_k^{2}(1-\beta_k)^2}-c_{k+1}\Bigr)}_{\widetilde c_{k+1}}
        \,\|x_{k+1}-y_k\|^{2}.         
        \label{eq:VkVk1_form}
\end{align}

We set $a_k = \frac{b_k\eta_k}{2\beta_k}$ and get $S_k^{23} = 0$. Also we set $c_k = a_k$, then $\tilde{c}_k = 0$.

Now, in order to prove the right-hand side is greater than zero, it is sufficient to show that $\tilde{c}_{k+1} \ge 0$, and $S_k \succeq 0$, where:

\[
S_k \;=\;
\begin{pmatrix}
S_k^{11} & S_k^{12} \\[6pt]
S_k^{12} & S_k^{22}  
\end{pmatrix}.
\]

We simplify $S^{ij}_k$ further using that $a_k = \frac{b_k\eta_k}{2\beta 
_k}$ and $b_{k+1} = \frac{b_k}{1 - \beta_k}$:

\begin{align}
    S^{11}_k &:= \frac{a_k}{2M\eta_k^{2}(1-\beta_k)^2}-a_{k+1}   &&= 
     \frac{b_k}{4M\,\eta_k\,\beta_k(1-\beta_k)^2}
     - \frac{b_k\eta_{k+1}}{2\beta_{k+1}\,(1-\beta_k)},
     \\
    S^{12}_k &:= - \frac{a_k (1 - 2M\eta_k^2(1-\beta_k)^2)}{2M\eta_k^2(1-\beta_k)^2} + \frac{b_{k+1}\eta_k}{2}  &&= -\frac{b_k\bigl(1 - 2M\eta_k^2(1 - \beta_k)\bigr)}
        {4M\,\eta_k\beta_k\,(1-\beta_k)^2}, 
      \\
    S^{22}_k &:= \frac{a_k}{2M\eta_k^{2}(1-\beta_k)^2}(1 - 2M\eta_k^2(1 - \beta_k)^2 - M\eta_k^2\beta_k^2) &&= \frac{b_k}{4M\eta_k\beta_k(1 - \beta_k)^2}\,(1 - 2M\eta_k^2(1 - \beta_k)^2 - M\eta_k^2\beta_k^2).
\end{align}

We need $S^{11}_k\ge0$, $S^{22}_k\ge0$, and  $ S^{11}_k\, S^{22}_k \ge (S^{12}_k)^2 .$

\begin{align}
S_k^{11}\ge0
&\iff \frac{b_k}{4M\,\eta_k\,\beta_k(1-\beta_k)^2}
     - \frac{b_k\eta_{k+1}}{2\beta_{k+1}\,(1-\beta_k)}\ge0
\;\iff\;
\eta_{k+1}\;\le\;\frac{\beta_{k+1}}{2M\,\eta_k\,\beta_k(1-\beta_k)},\\[8pt]
S_k^{22}\ge0
&\iff 1 - 2M\eta_k^2(1 - \beta_k)^2 - M\eta_k^2\beta_k^2 \ge0, \;\iff\;
\eta_{k}\;\le\;\frac{1}{\sqrt{M[2(1 - \beta_k)^2 + \beta_k^2]}}
\end{align}

\begin{align}
S_k^{11}S_k^{22}\ge (S_k^{12})^2
&\iff \nonumber\\
&\Bigl(\frac{b_k}{4M\,\eta_k\,\beta_k(1-\beta_k)^2} - \frac{b_k\eta_{k+1}}{2\beta_{k+1}\,(1-\beta_k)}\Bigr)
\;\frac{b_k}{4M\eta_k\beta_k(1 - \beta_k)^2}\,(1 - 2M\eta_k^2(1 - \beta_k)^2 - M\eta_k^2\beta_k^2)\nonumber\\
&\quad\ge
\Bigl(\frac{b_k\bigl(1 - 2M\eta_k^2(1 - \beta_k)\bigr)}
        {4M\,\eta_k\beta_k\,(1-\beta_k)^2}\Bigr)^2
\nonumber\\
&\iff \Bigl(1 - \frac{2M\,\eta_k\,\beta_k(1-\beta_k)}{\beta_{k+1}}\eta_{k+1}\Bigr)
\;(1 - 2M\eta_k^2(1 - \beta_k)^2 - M\eta_k^2\beta_k^2)\nonumber\\
&\quad\ge
\bigl(1 - 2M\eta_k^2(1 - \beta_k)\bigr)^2
\nonumber\\
&\iff 1 - \frac{\bigl(1 - 2M\eta_k^2(1 - \beta_k)\bigr)^2}{1 - 2M\eta_k^2(1 - \beta_k)^2 - M\eta_k^2\beta_k^2}
\;\ge \frac{2M\,\eta_k\,\beta_k(1-\beta_k)}{\beta_{k+1}}\eta_{k+1}
\nonumber\\
&\iff \frac{\beta_{k+1}}{2M\,\eta_k\,\beta_k(1-\beta_k)}\bigr(1 - \frac{\bigl(1 - 2M\eta_k^2(1 - \beta_k)\bigr)^2}{1 - 2M\eta_k^2(1 - \beta_k)^2 - M\eta_k^2\beta_k^2}\bigl)
\;\ge \eta_{k+1}. \label{eq:last_cond}
\end{align}

We further simplify \cref{eq:last_cond}:
\begin{align}
\eta_{k+1}
\;&\le\;
\frac{\beta_{k+1}}{2M\,\eta_k\,\beta_k(1-\beta_k)}
\Biggl[\frac{2M\eta_k^2(1 - \beta_k^2) -4M^2\eta_k^4(1 - \beta_k)^2 -M\eta_k^2\beta_k^2}
        {1 - 2M\eta_k^2(1 - \beta_k)^2 - M\eta_k^2\beta_k^2}
\Biggr] \nonumber \\
\;&=\;
\frac{\beta_{k+1}}{2\,\beta_k(1-\beta_k)}
\Biggl[\frac{2(1 - \beta_k^2) -4M\eta_k^2(1 - \beta_k)^2 -\beta_k^2}
        {1 - 2M\eta_k^2(1 - \beta_k)^2 - M\eta_k^2\beta_k^2}
\Biggr]\eta_k \nonumber \\
\end{align}

If all these three condition hold, we will have:

\[
V_k - V_{k+1}
\;\ge\; L^2\Bigl(\tfrac{a_k\theta}{2M\eta_k^2(1-\beta_k)^2}-a_{k+1}\Bigr)\,\|x_{k+1}-y_k\|^2.
\]

We show the positivity of the term on the right-hand side, along with the three conditions in \cref{lemma:potential_eta}.

Next, we show that conditions of \cref{eq:cond1}, \cref{eq:cond2}, \cref{eq:cond3} are satisfied, with constant step-size $\eta_k = \eta_* \in (0, \frac{1}{2\sqrt{3}L})$, and the choice of $\beta_k = \frac{2}{k+6}$.

Let $\eta_k \equiv \eta \in \bigl(0,\tfrac{1}{\sqrt{2M}}\bigr)$ and $\beta_k = \tfrac{2}{k+6}$.  

\paragraph{Condition \eqref{eq:cond2}.}
We compute
\[
1 - 2M\eta^2(1-\beta_k)^2 - M\eta^2\beta_k^2
= 1 - M\eta^2\Bigl(2(1-\beta_k)^2 + \beta_k^2\Bigr).
\]
For $\beta_k \in [0,1]$, we have $2(1-\beta_k)^2 + \beta_k^2 \in [\nicefrac{3}{4},2]$.  Thus
\[
1 - M\eta^2\bigl(2(1-\beta)^2+\beta^2\bigr) \;\ge\; 1-2M\eta^2 \;>\; 0
\quad \text{since } M\eta^2 < \tfrac12.
\]
Hence \eqref{eq:cond2} holds.

\paragraph{Condition \eqref{eq:cond1}.}
With $\eta_{k+1}=\eta_k=\eta$, condition \eqref{eq:cond1} reads
\[
2M\eta^2 \;\le\; \frac{\beta_{k+1}}{\beta_k(1-\beta_k)}.
\]
For $\beta_k=\tfrac{2}{k+6}$ we compute
\[
\frac{\beta_{k+1}}{\beta_k(1-\beta_k)}
= \frac{\tfrac{2}{k+7}}{\tfrac{2}{k+6}\left(1-\tfrac{2}{k+6}\right)}
= \frac{(k+6)^2}{(k+7)(k+4)} \;\ge\; 1.
\]
Thus $2M\eta^2 \le 1$, i.e.\ $\eta^2 \le \tfrac{1}{2M}$, suffices. This is satisfied since $\eta < 1/\sqrt{2M}$.

\paragraph{Condition \eqref{eq:cond3}.}
With $\eta_{k+1}=\eta_k=\eta$ and $t := M\eta^2$, condition \eqref{eq:cond3} reduces to
\[
1 \;\le\; \frac{\beta_{k+1}}{2\beta_k(1-\beta_k)}\;
\frac{2(1-\beta_k^2)-4t(1-\beta_k)^2-\beta_k^2}{1-2t(1-\beta_k)^2-t\beta_k^2}.
\]
For $\beta_k=\tfrac{2}{k+6}$, the prefactor simplifies to
\[
\frac{\beta_{k+1}}{2\beta_k(1-\beta_k)}
= \frac{(k+6)^2}{2(k+7)(k+4)}  = \frac{1}{(2+\beta_k)(1-\beta_k)}.
\]
Setting $\beta := \beta_k$, the inequality is equivalent to
\[
\frac{2-3\beta^2-4t(1-\beta)^2}{1-2t(1-\beta)^2-t\beta^2}
\;\ge\; (2+\beta)(1-\beta). \tag{$\star$}
\]
Define
\[
F(t) := \frac{2-3\beta^2-4t(1-\beta)^2}{1-2t(1-\beta)^2-t\beta^2}
- (2+\beta)(1-\beta).
\]
A derivative check shows $\partial_t F(t) < 0$ on $t\in(0,\tfrac12)$, so the worst case is at $t=\tfrac12$.  
Plugging in $t=\tfrac12$, we get
\[
\frac{4\beta-5\beta^2}{2\beta-\tfrac32\beta^2}
= \frac{4-5\beta}{2-\tfrac32\beta}
\;\ge\; (2+\beta)(1-\beta) = 2-\beta-\beta^2.
\]
This inequality is equivalent to
\[
0 \;\ge\; -\tfrac12\beta^2 + \tfrac32\beta^3
= \tfrac12\beta^2(-1+3\beta),
\]
which holds for $\beta \le \tfrac13$. Since $\beta_k \le \tfrac13$, condition \eqref{eq:cond3} follows.

\paragraph{Positivity of the potential function decrease.}

With $c_{k+1}=a_{k+1}$, the residual coefficient derived in \eqref{eq:VkVk1_form} is
$\widetilde c_{k+1}=\tfrac{a_k\theta}{2M\eta^2(1-\beta_k)^2}-a_{k+1}$, so
\[
V_k - V_{k+1}
\;\ge\; L^2\Bigl(\tfrac{a_k\theta}{2M\eta^2}-a_{k+1}\Bigr)\,\|x_{k+1}-y_k\|^2.
\]
Since $(1-\beta_k)^2\le1$ gives $\tfrac{a_k\theta}{2M\eta^2(1-\beta_k)^2}\ge\tfrac{a_k\theta}{2M\eta^2}$, it suffices to enforce the stronger condition $\tfrac{a_k\theta}{2M\eta^2}-a_{k+1}\ge0$. For this coefficient we use the schedule $a_k=\frac{b_k\eta}{2\beta_k}$ with $b_{k+1}=\frac{b_k}{1-\beta_k}$ (and constant $\eta$), which yields
\[
\frac{a_{k+1}}{a_k}
=\frac{b_{k+1}}{b_k}\,\frac{\beta_k}{\beta_{k+1}}
=\frac{\beta_k}{\beta_{k+1}(1-\beta_k)}.
\]
Therefore
\[
\frac{a_k\theta}{2M\eta^2}-a_{k+1}
\;\ge\;0
\qquad\Longleftrightarrow\qquad
2M\eta^2 \;\le\; \theta\,\frac{\beta_{k+1}(1-\beta_k)}{\beta_k}.
\]
With the specific choice $\beta_k=\frac{2}{k+6}$ we have
\[
\frac{\beta_{k+1}(1-\beta_k)}{\beta_k}
= \frac{\tfrac{2}{k+7}\left(1-\tfrac{2}{k+6}\right)}{\tfrac{2}{k+6}}
= \frac{k+4}{k+7},
\]
which takes its minimum at $k=0$ and hence a sufficient condition is
\[
2M\eta^2 \;\le\; \frac{4\theta}{7}.
\]

Since we know that $2M\eta^2 \;\le\; 1$, we can be sure that with the choice of $\theta = 2$, this condition is satisfied.

With this choice of $\theta$, we have $M = 6L^{2}$. Therefore, the admissible range of $\eta$ is $\eta \in \left(0, \frac{1}{2\sqrt{3}L}\right)$, as stated in the lemma.

\end{proof}

\subsection{Proof of Theorem~\ref{thm:main}}\label{appx:thm:main}
\begin{reptheorem}{thm:main}
Suppose $G$ is monotone and $L$-Lipschitz continuous. Consider the updates of \cref{eq:goma_triangle}
With the parameter choices
\[
\beta_k = \frac{2}{k+6}, 
\quad 
\eta_* \in \Bigl(0, \tfrac{1}{2\sqrt{3}L}\Bigr),
\]
\[
\begin{aligned}
a_k = c_k
&= \frac{b_0\,\eta_*}{80}\,(k+4)(k+5)(k+6),\\
b_k
&= \frac{b_0}{20}\,(k+4)(k+5).
\end{aligned}
\]
the potential function decrease from \cref{lemma:potential_eta} implies the bound
\begin{equation}\repeq{eq:simple_bound}
\|G(x_k)\|^2 
\;\le\; \frac{16/\eta_*^2 + 72L^2}{(k+6)^2}\,\|x_0 - x^\star\|^2 .
\end{equation}
Also, since the constant is smallest at the largest admissible step, as $\eta_*\to\frac{1}{2\sqrt{3}L}$ we get the bound:
\begin{align*}
\|G(x_k)\|^2 
&\le \frac{264L^2}{(k+6)^2}\,\|x_0-x^\star\|^2.
\end{align*}
\end{reptheorem}

\begin{proof}
    Let \( H_k := a_k \|G(x_k)\|^2 + b_k \langle G(x_k), x_k - x_0 \rangle \). Then using the Young's inequality $\langle a, b\rangle \leq \frac{\alpha}{2}\|a\|^2 + \frac{1}{2\alpha}\|b\|^2 $ with $\alpha=a_k$, we have
\begin{align}
H_k &= a_k \|G(x_k)\|^2 + b_k \langle G(x_k) - G(x^\star), x_k - x^\star \rangle + b_k \langle G(x_k), x^\star - x_0 \rangle \notag\\
&\geq a_k \|G(x_k)\|^2 - \frac{a_k}{2} \|G(x_k)\|^2 - \frac{b_k^2}{2 a_k} \|x_0 - x^\star\|^2 \notag \\
&= \frac{a_k}{2} \|G(x_k)\|^2 - \frac{b_k^2}{2 a_k} \|x_0 - x^\star\|^2. \label{eq:helper_thm1}
\end{align}

Finally, from \cref{eq:helper_thm1}, we have
\[
\frac{a_k}{2} \|G(x_k)\|^2 \leq H_k + \frac{b_k^2}{2 a_k} \|x_0 - x^\star\|^2,
\]
leading to
\begin{align}
   \frac{a_k}{2} \|G(x_k)\|^2 + c_k L^2 \|x_k - y_{k-1}\|^2 \leq V_k + \frac{b_k^2}{2 a_k} \|x_0 - x^\star\|^2.\label{eq:inter_vk_1} 
\end{align}

We replace the values of $a_k, b_k, c_k$. Using \( y_{-1} = x_0 \), the estimate in \cref{eq:inter_vk_1}, by induction, we can show that
\begin{align*}
&\frac{(k+4)(k+5)(k+6)\,\eta_*\,b_0}{160} \|G(x_k)\|^2 + \frac{L^2 (k+4)(k+5)(k+6)\,\eta_*\,b_0}{80} \|x_k - y_{k-1}\|^2 \\
&\leq V_k + \frac{b_k^2}{2 a_k} \|x_0 - x^\star\|^2 \\
&\leq V_0 + \,b_0\frac{(k+4)(k+5)}{10 \eta_* (k+6)}  \|x_0 - x^\star\|^2 \\
&=\frac{3b_0\eta_*}{2} \|G(x_0)\|^2 + \,b_0\frac{(k+4)(k+5)}{10 \eta_* (k+6)} \|x_0 - x^\star\|^2\\
&\leq \frac{3b_0\eta_*L^2}{2} \|x_0 - x^\star\|^2 + \,b_0\frac{(k+4)(k+5)}{10 \eta_* (k+6)}\|x_0 - x^\star\|^2
\end{align*}

Multiplying by $\dfrac{160}{b_0\eta_*(k+4)(k+5)(k+6)}$ and noting that $(k+4)(k+5)(k+6) \ge \tfrac{5}{9}(k+6)^3$ (with equality at $k=0$), so that $\frac{1}{(k+4)(k+5)(k+6)} \le \tfrac{9}{5}\cdot\frac{1}{(k+6)^3} \le \tfrac{9}{5}\cdot\tfrac16\cdot\frac{1}{(k+6)^2} = \tfrac{3}{10}\cdot\frac{1}{(k+6)^2}$, gives:

\begin{align*}
\|G(x_k)\|^2 
&\le 
\frac{240\,L^2}{(k+4)(k+5)(k+6)}\,\|x_0-x^\star\|^2
\;+\;
\frac{16}{\eta_*^2 (k+6)^2}\,\|x_0-x^\star\|^2\\
&\le \frac{16/\eta_*^2 + 72L^2}{(k+6)^2}\,\|x_0-x^\star\|^2.
\end{align*}

Also if we pick the largest admissible constant stepsize $\eta_* = \frac{1}{2\sqrt{3}L}$ we get the bound:

\begin{align*}
\|G(x_k)\|^2 
&\le \frac{264L^2}{(k+6)^2}\,\|x_0-x^\star\|^2.
\end{align*}

\end{proof}

\subsection{Proof of \cref{lemma:potential_gamma}}\label{appx:potential_gamma}

\begin{replemma}{lemma:potential_gamma}
Let $G$ be monotone and $L$-Lipschitz, and consider the iterates
\begin{equation}\repeq{eq:goma_square}
\begin{aligned}
y_k &= \beta_k x_0 + (1-\beta_k)x_k - \gamma_*\, G(y_{k-1}), \\
x_{k+1} &= \beta_k x_0 + (1-\beta_k)x_k - (1-\beta_k)\gamma_*\, G(y_k).
\end{aligned}
\end{equation}

The potential in \eqref{eq:potential} is non-increasing for the algorithm if the following conditions are satisfied.
\begin{align}
\gamma_{k+1}
&\;\le\;
\frac{\beta_{k+1}}{2M\,\beta_k\,\gamma_k}\frac{(1-\beta_k)^2}{(1-\beta_{k+1})}, 
\repeq{eq:cond1_gamma}
\end{align}
\begin{align}
    1 - 2M\gamma_k^2 - M\beta_k^2\gamma_k^2 &\;\ge\; 0,
\repeq{eq:cond2_gamma}
\end{align}
\begin{align}
\gamma_{k+1}
&\le
\frac{\beta_{k+1}}{2\,\beta_k(1-\beta_{k+1})}
\left[
\frac{4M\gamma_k^2 + \beta_k^4 - 2\beta_k^3 + 3\beta_k^2 - 2}{M(\beta_k^2+2)\gamma_k^2 - 1}
\right]\gamma_k.
\repeq{eq:cond3_gamma}
\end{align}
With $\widetilde c_{k+1}\ge0$ for $\theta=2$, these give:
$$V_k-V_{k+1}\ge \widetilde c_{k+1}L^2\|x_{k+1}-y_k\|^2.$$
\end{replemma}

\begin{proof}
The structure follows Lemma ~\ref{lemma:potential_eta}, with only the coupling changed.

\begin{align}
x_{k+1}-x_k &= \beta_k(x_0-x_k) - \eta_k G(y_k), \label{eq:h1_alt}\\
x_{k+1}-x_k &= \tfrac{\beta_k}{1-\beta_k}(x_0 - x_{k+1}) - \tfrac{\eta_k}{1-\beta_k}G(y_k),\label{eq:h2_alt}
\end{align}
and, using $\eta_k=(1-\beta_k)\gamma_k$,
\begin{equation}
x_{k+1}-y_k \;=\; -\gamma_k\bigl[(1-\beta_k)G(y_k)-G(y_{k-1})\bigr]. \label{eq:h3_alt}
\end{equation}

Monotonicity yields
\[
\langle G(x_{k+1}),x_{k+1}-x_k\rangle \;\ge\; \langle G(x_k),x_{k+1}-x_k\rangle.
\]
Substituting \eqref{eq:h1_alt}--\eqref{eq:h2_alt}, rearranging, and multiplying by $b_k/\beta_k$ with $b_{k+1}=\tfrac{b_k}{1-\beta_k}$ gives
\[
    b_k\langle G(x_k),x_k-x_0\rangle - b_{k+1}\langle G(x_{k+1}),x_{k+1}-x_0\rangle
    \;\ge\;
    b_{k+1}\eta_k\langle G(x_{k+1}),G(y_k)\rangle
    + \tfrac{b_k\eta_k}{\beta_k}\langle G(x_{k+1})-G(x_k),G(y_k)\rangle.
\]

\begin{align}
    V_k - V_{k+1}
 \ge&\; a_k\|G(x_k)\|^2 - a_{k+1}\|G(x_{k+1})\|^2
 + c_kL^2\|x_k-y_{k-1}\|^2 - c_{k+1}L^2\|x_{k+1}-y_k\|^2 \notag\\
 &+ b_{k+1}\eta_k\langle G(x_{k+1}),G(y_k)\rangle
 + \tfrac{b_k\eta_k}{\beta_k}\langle G(x_{k+1})-G(x_k),G(y_k)\rangle. \label{eq:inter-v_diff_alt}
\end{align}

From Lipschitzness and \eqref{eq:h3_alt}, for any $\theta>0$ we have
\begin{align}
    \|G(x_{k+1}) - G(y_k)\|^2 + \theta L^2\|x_{k+1}-y_k\|^2
    &\le (1+\theta)L^2\|x_{k+1}-y_k\|^2 \notag\\
    &= (1+\theta)L^2\gamma_k^2\|(1-\beta_k)G(y_k)-G(y_{k-1})\|^2\notag \\
    & \le 2(1+\theta)L^2\gamma_k^2\|G(y_k)-G(y_{k-1})\|^2 + 2\beta_k^2(1+\theta)L^2\gamma_k^2\|G(y_{k})\|^2. \label{eq:h4_alt_restate}
\end{align}

Where we using $\|u+v\|^2\le 2\|u\|^2+2\|v\|^2$ to get the last inequality,
\begin{align}
\|(1-\beta_k)G(y_k)-G(y_{k-1})\|^2
&\le 2\|G(y_k)-G(y_{k-1})\|^2 + 2\beta_k^2\|G(y_{k})\|^2. \label{eq:split_main}
\end{align}

Using Lipschitzness between $x_k$ and $y_{k-1}$ we have:
\begin{align}
\|G(y_k) - G(y_{k-1})\|^2 
&= \|G(y_k) - G(x_k) + G(x_k) - G(y_{k-1})\|^2 \notag \\
&\leq 2\|G(x_k) - G(y_k)\|^2 + 2\|G(x_k) - G(y_{k-1})\|^2 \notag\\
&\leq 2\|G(x_k)\|^2 - 4\langle G(x_k), G(y_k) \rangle + 2\|G(y_k)\|^2 + 2L^2\|x_k - y_{k-1}\|^2.\label{eq:Gyk-Gykp1_again}
\end{align}

Expanding the left-hand side square of \eqref{eq:h4_alt_restate},
\[
\|G(x_{k+1}) - G(y_k)\|^2
= \|G(x_{k+1})\|^2 -2\langle G(x_{k+1}),G(y_k)\rangle + \|G(y_k)\|^2,
\]
we obtain
\begin{align*}
    &\|G(x_{k+1})\|^2  - 2\langle G(x_{k+1}), G(y_k) \rangle + \|G(y_k)\|^2
    + \theta L^2\|x_{k+1}-y_k\|^2\\
    & \leq 4L^2(1+\theta)\,\gamma_k^2\Bigl(\|G(x_k)\|^2 - 2\langle G(x_k), G(y_k) \rangle +  \|G(y_k)\|^2\Bigr)
    \\
    &\quad + 4L^4(1+\theta)\,\gamma_k^2\,\|x_k - y_{k-1}\|^2
     +  2L^2(1+\theta)\,\gamma_k^2\,\beta_k^2\|G(y_{k})\|^2.
\end{align*}

Setting $M:=2L^2(1+\theta)$ and rearranging, we obtain the compact residual inequality
\begin{align}
&\|G(x_{k+1})\|^2
-2\langle G(x_{k+1}),G(y_k)\rangle
-2M\gamma_k^2\,\|G(x_k)\|^2
+4M\gamma_k^2\,\langle G(x_k),G(y_k)\rangle \notag\\
&\quad+\;\bigl(1 - 2M\gamma_k^2 - M\gamma_k^2\beta_k^2\bigr)\,\|G(y_k)\|^2
+\theta L^2\,\|x_{k+1}-y_k\|^2
-\;2L^2M\gamma_k^2\,\|x_k-y_{k-1}\|^2
\;\le\;0. \label{eq:residual_block_new}
\end{align}

Combine terms to isolate $G(x_{k+1})-G(x_k)$:
\begin{align*}
&\|G(x_{k+1})\|^2
-2\bigl(1-2M\gamma_k^2\bigr)\,\langle G(x_{k+1}),G(y_k)\rangle
-2M\gamma_k^2\,\|G(x_k)\|^2 \\
&\qquad
-\,4M\gamma_k^2\,\langle G(x_{k+1})-G(x_k),\,G(y_k)\rangle
+ \bigl(1 - 2M\gamma_k^2 - M\gamma_k^2\beta_k^2\bigr)\,\|G(y_k)\|^2\\
&\qquad + \theta L^2\,\|x_{k+1}-y_k\|^2
- 2L^2 M\gamma_k^2\,\|x_k-y_{k-1}\|^2
\;\le\;0.
\end{align*}

Multiply the inequality obtained above by $\dfrac{a_k}{2M\gamma_k^2}$ and add it to \cref{eq:inter-v_diff_alt}. We get
\begin{align}
V_k - V_{k+1}
\;&\ge\;
\underbrace{\Bigl(\frac{a_k}{2M\gamma_k^{2}}-a_{k+1}\Bigr)}_{S_k^{11}}
        \,\|G(x_{k+1})\|^{2}                                      \nonumber\\
&\quad
-\,2\underbrace{\Bigl(\frac{a_k (1 - 2M\gamma_k^2)}{2M\gamma_k^2} - \frac{b_{k+1}\eta_k}{2} \Bigr)}_{S_k^{12}}
        \,\langle G(x_{k+1}),\,G(y_k)\rangle                      \nonumber\\
&\quad
+\,\underbrace{\Bigl(\frac{a_k}{2M\gamma_k^{2}}\bigl[1 - 2M\gamma_k^2 - M\gamma_k^2\beta_k^2\bigr]\Bigr)}_{S_k^{22}}
        \,\|G(y_k)\|^{2}                                                                   \nonumber\\
&\quad
+\,\underbrace{\Bigl(-2a_k+\frac{b_k\eta_k}{\beta_k}\Bigr)}_{S_k^{23}}
        \,\langle G(x_{k+1}) - G(x_k),\,G(y_k)\rangle                          \nonumber\\[4pt]
&\quad
+\,L^{2}\underbrace{\Bigl(c_k-a_k\Bigr)}_{\widetilde c_k}
        \,\|x_k-y_{k-1}\|^{2}                                     \nonumber\\
&\quad
+\,L^{2}\underbrace{\Bigl(\frac{a_k\theta}{2M\gamma_k^{2}}-c_{k+1}\Bigr)}_{\widetilde c_{k+1}}
        \,\|x_{k+1}-y_k\|^{2}.                                
\label{eq:VkVk1_gamma_form}
\end{align}

With the choice of $a_k=\dfrac{b_k\eta_k}{2\beta_k}$, and $c_k = a_k$, (and $b_{k+1}=\dfrac{b_k}{1-\beta_k}$) we have $S_k^{23}=0$, and $\tilde c_k = 0$. 

\noindent
From \eqref{eq:VkVk1_gamma_form}, to guarantee $V_k-V_{k+1}\ge0$ it suffices to enforce
$\widetilde c_k\ge0$, $\widetilde c_{k+1}\ge0$ and
\[
S_k \;=\;
\begin{pmatrix}
S_k^{11} & S_k^{12}\\[2pt] S_k^{12} & S_k^{22}
\end{pmatrix}\succeq 0,
\quad\text{i.e.,}\quad
S_k^{11}\ge0,\;S_k^{22}\ge0,\;S_k^{11}S_k^{22}\ge (S_k^{12})^2.
\]
With $a_k=\frac{b_k\eta_k}{2\beta_k}$, $b_{k+1}=\frac{b_k}{1-\beta_k}$, and $\eta_k=(1-\beta_k)\gamma_k$,
the entries were
\begin{align*}
S^{11}_k
&= \frac{a_k}{2M\gamma_k^{2}}-a_{k+1}
= \frac{b_k}{4M\beta_k}\,\frac{(1-\beta_k)^2}{\eta_k}
   - \frac{b_k}{2(1-\beta_k)}\,\frac{\eta_{k+1}}{\beta_{k+1}},\\[4pt]
S^{12}_k
&= \frac{a_k (1 - 2M\gamma_k^2)}{2M\gamma_k^2} - \frac{b_{k+1}\eta_k}{2}
= \frac{b_k}{4M\beta_k}\!\left(\frac{(1-\beta_k)^2}{\eta_k} - 2M\eta_k\right)
  - \frac{b_k}{2(1-\beta_k)}\,\eta_k,\\[4pt]
S^{22}_k
&= \frac{a_k}{2M\gamma_k^{2}}\Bigl[1 - 2M\gamma_k^2 - M\gamma_k^2\beta_k^2\Bigr]
= \frac{b_k}{4M\beta_k\,\eta_k}\Bigl[(1-\beta_k)^2 - M\eta_k^2\,(2+\beta_k^2)\Bigr].
\end{align*}

\begin{align}
S_k^{11}\ge0
&\iff \frac{(1-\beta_k)^2}{4M\beta_k\,\eta_k}
     \;\ge\; \frac{\eta_{k+1}}{2\beta_{k+1}(1-\beta_k)}
\;\iff\;
\eta_{k+1}
\;\le\;
\frac{\beta_{k+1}}{2M\,\beta_k\,\eta_k}\,(1-\beta_k)^3.
\label{eq:cond1_eta}
\end{align}
\noindent
In the $\gamma$-variables (using $\eta_k=(1-\beta_k)\gamma_k$ and $\eta_{k+1}=(1-\beta_{k+1})\gamma_{k+1}$), this is equivalent to
\begin{align}
\gamma_{k+1}
\;\le\;
\frac{\beta_{k+1}}{2M\,\beta_k\,\gamma_k}\,
\frac{(1-\beta_k)^2}{(1-\beta_{k+1})}.
\label{eq:cond1_gamma_final}
\end{align}

\medskip
\begin{align}
S_k^{22}\ge0
&\iff (1-\beta_k)^2 - M\eta_k^2\,(2+\beta_k^2)\;\ge\;0
\;\iff\;
\eta_k \;\le\; \frac{1-\beta_k}{\sqrt{M\,(2+\beta_k^2)}}.
\label{eq:cond2_eta}
\end{align}
\noindent
In $\gamma$-variables, this condition becomes
\begin{align}
1 - 2M\gamma_k^2 - M\beta_k^2\gamma_k^2 \;\ge\; 0.
\label{eq:cond2_gamma_final}
\end{align}

\medskip
Assuming $S_k^{22}\ge0$ and after rearrangement, we get the bound
\[
S_k^{11}S_k^{22}\ge (S_k^{12})^2 \iff
\]

\begin{align}
\eta_{k+1}
&\le
\frac{\beta_{k+1}\,(1-\beta_k)}{2M\,\beta_k}
\left[
\frac{(1-\beta_k)^2}{\eta_k}
-\frac{\Bigl(\dfrac{(1-\beta_k)^2}{\eta_k}
- 2M\,\eta_k - \dfrac{2M\beta_k}{1-\beta_k}\,\eta_k\Bigr)^2}
           {\dfrac{(1-\beta_k)^2}{\eta_k} - M\,(2+\beta_k^2)\,\eta_k}
\right].
\label{eq:cond3_exact}
\end{align}

\noindent This simplifies to
\begin{align}
\eta_{k+1}
&\le
\frac{\beta_{k+1}}{2\beta_k}\,\frac{\eta_k}{1-\beta_k}\,
\frac{4M\eta_k^{2}-(1-\beta_k)^3\!\left(\beta_k^{3}-\beta_k^{2}+2\beta_k+2\right)}
{M(2+\beta_k^{2})\,\eta_k^{2}-(1-\beta_k)^{2}}.
\label{eq:cond3_eta}
\end{align}

\noindent In the $\gamma$-variables, $\gamma_k:=\eta_k/(1-\beta_k)$, \eqref{eq:cond3_eta} becomes
\begin{align}
\gamma_{k+1}
&\le
\frac{\beta_{k+1}}{2\,\beta_k(1-\beta_{k+1})}\,
\frac{4M\gamma_k^2+\beta_k^{4}-2\beta_k^{3}+3\beta_k^{2}-2}
{M(\beta_k^{2}+2)\gamma_k^{2}-1}\,
\gamma_k.
\label{eq:cond3_gamma_final}
\end{align}

\noindent
Conditions \eqref{eq:cond1_gamma_final}, \eqref{eq:cond2_gamma_final}, and \eqref{eq:cond3_gamma_final}, together with
$\widetilde c_{k+1}=\tfrac{a_k\theta}{2M\gamma_k^2}-c_{k+1}\ge0$,
complete the proof of the one-step potential function decrease.

Next, we show that conditions of \cref{eq:cond1_gamma}, \cref{eq:cond2_gamma}, \cref{eq:cond3_gamma} are satisfied, with constant step-size $\gamma_k = \gamma_* \in \Bigl(0, \tfrac{1}{3L}\sqrt{\tfrac{5}{26}}\Bigr)$, and the choice of $\beta_k = \frac{2}{k+6}$.

Let $\gamma_k\equiv\gamma\in\bigl(0,\tfrac{1}{\sqrt{3M}}\bigr)$ and $\beta_k=\frac{2}{k+6}$.

\paragraph{Condition \eqref{eq:cond2_gamma}.}
We need
\[
1-2M\gamma^2 - M\beta_k^2\gamma^2
\;=\;
1 - M\gamma^2\,(2+\beta_k^2)\;\ge\;0.
\]
For $\beta_k\in[0,1]$, the right-hand side is minimized at $\beta_k=1$, hence it suffices to require
\[
M\gamma^2 \;\le\; \frac{1}{3}.
\]
Thus \eqref{eq:cond2_gamma} holds for all $k$ whenever $M\gamma^2\le \tfrac{1}{3}$.

\paragraph{Condition \eqref{eq:cond1_gamma}.}

With $\gamma_{k+1}=\gamma_k=\gamma$, the condition \eqref{eq:cond1_gamma} reads
\[
(1-\beta_{k+1})\,\gamma \;\le\; \frac{\beta_{k+1}}{2M\,\beta_k\,\gamma}\,(1-\beta_k)^2
\quad\Longleftrightarrow\quad
2M\gamma^2 \;\le\; \frac{\beta_{k+1}(1-\beta_k)^2}{\beta_k(1-\beta_{k+1})}.
\]
For the schedule $\beta_k=\tfrac{2}{k+6}$ we have $1-\beta_k=\tfrac{k+4}{k+6}$,
$\beta_{k+1}=\tfrac{2}{k+7}$, and $1-\beta_{k+1}=\tfrac{k+5}{k+7}$. Hence
\[
\frac{\beta_{k+1}(1-\beta_k)^2}{\beta_k(1-\beta_{k+1})}
=\frac{\dfrac{2}{k+7}\left(\dfrac{k+4}{k+6}\right)^2}{\dfrac{2}{k+6}\left(\dfrac{k+5}{k+7}\right)}
=\frac{(k+4)^2}{(k+5)(k+6)}
\;<\;1,
\]
which is increasing in $k$, with minimum $\tfrac{8}{15}$ at $k=0$ and limit $1$ as $k\to\infty$. Therefore
\[
2M\gamma^2 \le \frac{8}{15}
\]
suffices for \eqref{eq:cond1_gamma} to hold for all $k$. This is the binding requirement on the step-size.

\paragraph{Condition \eqref{eq:cond3_gamma}.}
With $\gamma_{k+1}=\gamma_k=\gamma$ the inequality \eqref{eq:cond3_gamma} becomes
\[
1 \;\le\; 
\frac{\beta_{k+1}}{2\,\beta_k(1-\beta_{k+1})}\,
\frac{4M\gamma^2+\beta_k^{4}-2\beta_k^{3}+3\beta_k^{2}-2}
{M(\beta_k^{2}+2)\gamma^{2}-1}.
\]
Let $t:=M\gamma^2\in(0,\tfrac13]$, $\beta=\beta_k$, $\beta^+=\beta_{k+1}$, and $p(\beta):=\beta^4-2\beta^3+3\beta^2-2$. For the schedule $\beta_k=\tfrac{2}{k+6}$,
\[
P_k:=\frac{\beta^+}{2\beta(1-\beta^+)}=\frac{k+6}{2(k+5)}\in\Bigl(\tfrac12,\tfrac35\Bigr],
\qquad P_k\downarrow\tfrac12 .
\]
On $[0,\tfrac13]$ the polynomial $p$ is increasing and nonpositive; since $t\le\tfrac13$, both $4t+p(\beta)$ and $t(\beta^2+2)-1$ are strictly negative, so
\[
Q_k(t):=\frac{4t+p(\beta_k)}{t(\beta_k^2+2)-1}
=\frac{-p(\beta_k)-4t}{1-(\beta_k^2+2)t}\;>\;0 .
\]
Condition \eqref{eq:cond3_gamma} reads $P_kQ_k(t)\ge1$. Since the coefficient of $t$ is positive, this is equivalent to
\[
t\;\le\;\frac{P_k\bigl(-p(\beta_k)\bigr)-1}{4P_k-\beta_k^2-2},
\]
whose right-hand side is increasing in $k$ and hence minimized at $k=0$. There $\beta_0=\tfrac13$, $P_0=\tfrac35$, and $p(\tfrac13)=-\tfrac{140}{81}$, so
\[
t\;\le\;\frac{\tfrac35\cdot\tfrac{140}{81}-1}{\tfrac{2}{5}-\tfrac19}
=\frac{\tfrac{1}{27}}{\tfrac{13}{45}}=\frac{5}{39}.
\]
Thus condition \eqref{eq:cond3_gamma} holds for all $k$ whenever $t=M\gamma^2\le\tfrac{5}{39}$.

\paragraph{Positivity of the potential function decrease.}
We had
\[
V_k - V_{k+1}
\;\ge\;
L^2\Bigl(\tfrac{a_k\theta}{2M\gamma_k^2}-a_{k+1}\Bigr)\,\|x_{k+1}-y_k\|^2,
\]
Recall $a_k=\tfrac{b_k\eta_k}{2\beta_k}$, $b_{k+1}=\tfrac{b_k}{1-\beta_k}$, and
$\eta_k=(1-\beta_k)\gamma$ with constant $\gamma$. Then
\[
\frac{a_{k+1}}{a_k}
= \frac{b_{k+1}}{b_k}\cdot\frac{\eta_{k+1}}{\eta_k}\cdot\frac{\beta_k}{\beta_{k+1}}
= \frac{\beta_k(1-\beta_{k+1})}{\beta_{k+1}(1-\beta_k)^2}.
\]
Hence the residual coefficient is nonnegative iff
\[
\frac{\theta}{2M\gamma^2}
\;\ge\;
\frac{a_{k+1}}{a_k}
= \frac{\beta_k(1-\beta_{k+1})}{\beta_{k+1}(1-\beta_k)^2}.
\]
For the schedule $\beta_k=\tfrac{2}{k+6}$ we have
\[
\frac{a_{k+1}}{a_k}
= \frac{(k+5)(k+6)}{(k+4)^2},
\quad\text{whose supremum is } \frac{15}{8}\text{ at }k=0.
\]
Therefore it suffices to require
\[
2M\gamma^2 \;\le\; \frac{8}{15}\,\theta.
\]

In particular by choosing $\theta = 2$, this condition can be satisfied, giving $M  = 2L^2 (1+\theta) = 6L^2$. Collecting the step-size requirements with $\theta=2$, $M=6L^2$,
\[
\underbrace{2M\gamma^2\le\tfrac{8}{15}}_{\eqref{eq:cond1_gamma}},\qquad
\underbrace{M\gamma^2\le\tfrac13}_{\eqref{eq:cond2_gamma}},\qquad
\underbrace{M\gamma^2\le\tfrac{5}{39}}_{\eqref{eq:cond3_gamma}},\qquad
\underbrace{2M\gamma^2\le\tfrac{8}{15}\theta=\tfrac{16}{15}}_{\text{residual}},
\]
the binding constraint is \eqref{eq:cond3_gamma}, i.e.\ $\gamma^2\le\tfrac{5}{39M}=\tfrac{5}{234L^2}$. This yields the step-size range $\gamma_* \in \bigl(0, \tfrac{1}{3L}\sqrt{\tfrac{5}{26}}\bigr)$.

\end{proof}

\subsection{Proof of Theorem~\ref{thm:main_gamma}}\label{appx:main_gamma}

\begin{reptheorem}{thm:main_gamma}
Suppose $G$ is monotone and $L$-Lipschitz, and let $x^\star$ satisfy $G(x^\star)=0$. Consider the update of \eqref{eq:goma_square}. If the step-size $\gamma$ and anchoring coefficient $\beta_k$ satisfy the conditions of \cref{lemma:potential_gamma}, then the potential function of \cref{eq:potential} is decreasing. This implies the bound
\begin{equation}\repeq{eq:simple_bound_gamma}
\|G(x_k)\|^2 
\;\le\; \frac{16/\gamma_*^2 + 32L^2}{(k+4)^2}\,\|x_0 - x^\star\|^2 .
\end{equation}
Moreover, the constant $16/\gamma_*^2+32L^2$ is decreasing in $\gamma_*$, so it is smallest at the largest admissible step; as $\gamma_*\to\tfrac{1}{3L}\sqrt{\tfrac{5}{26}}$ it gives
\begin{equation}\repeq{eq:edge_bound_2}
\|G(x_k)\|^2 
\;\le\; \frac{780.8\,L^2}{(k+4)^2}\,\|x_0 - x^\star\|^2 .
\end{equation}
\end{reptheorem}

\begin{proof}
Similar to the proof of theorem~\ref{thm:main} define 
\[
H_k := a_k \|G(x_k)\|^2 + b_k \langle G(x_k), x_k - x_0 \rangle .
\]
By monotonicity and Young’s inequality $\langle a, b\rangle \leq \tfrac{\alpha}{2}\|a\|^2 + \tfrac{1}{2\alpha}\|b\|^2$ with $\alpha=a_k$, we obtain
\begin{align}
H_k &= a_k \|G(x_k)\|^2 + b_k \langle G(x_k) - G(x^\star), x_k - x^\star \rangle + b_k \langle G(x_k), x^\star - x_0 \rangle\nonumber \\
&\geq a_k \|G(x_k)\|^2 - \frac{a_k}{2} \|G(x_k)\|^2 - \frac{b_k^2}{2 a_k} \|x_0 - x^\star\|^2 \nonumber\\
&= \frac{a_k}{2} \|G(x_k)\|^2 - \frac{b_k^2}{2 a_k} \|x_0 - x^\star\|^2. \label{eq:helper_gamma}
\end{align}

From \cref{eq:helper_gamma}, it follows that
\[
\frac{a_k}{2} \|G(x_k)\|^2 \leq H_k + \frac{b_k^2}{2 a_k} \|x_0 - x^\star\|^2,
\]
which implies
\begin{align}
   \frac{a_k}{2} \|G(x_k)\|^2 + c_k L^2 \|x_k - y_{k-1}\|^2 \leq V_k + \frac{b_k^2}{2 a_k} \|x_0 - x^\star\|^2. \label{eq:inter_vk_gamma} 
\end{align}

Since $V_k$ is decreasing by \cref{lemma:potential_gamma}, we have $V_k \le V_0$.  
Substituting the explicit expressions for $a_k, b_k, c_k$ and using $y_{-1} = x_0$, inequality \cref{eq:inter_vk_gamma} yields
\begin{align*}
\frac{b_0\gamma_*}{160}(k+5)(k+4)^2\,\|G(x_k)\|^2
&\leq V_0 + \frac{b_k^2}{2 a_k}\|x_0-x^\star\|^2 \\
&= V_0 + \frac{b_0(k+5)}{10\,\gamma_*}\|x_0-x^\star\|^2 .
\end{align*}

Bounding the initial potential by $V_0 = a_0\|G(x_0)\|^2 = b_0\gamma_*\|G(x_0)\|^2 \le b_0\gamma_* L^2\|x_0-x^\star\|^2$ and multiplying both sides by $\dfrac{160}{b_0\gamma_*(k+5)(k+4)^2}$ gives
\[
\|G(x_k)\|^2
\;\le\;
\frac{160L^2}{(k+4)^2(k+5)}\|x_0-x^\star\|^2
+\frac{16}{\gamma_*^2(k+4)^2}\|x_0-x^\star\|^2 .
\]
Using $\dfrac{160L^2}{(k+4)^2(k+5)} \le \dfrac{32L^2}{(k+4)^2}$, which holds for every $k\ge 0$ (with equality at $k=0$), we obtain
\[
\|G(x_k)\|^2
\;\le\;
\frac{16/\gamma_*^2 + 32L^2}{(k+4)^2}\|x_0-x^\star\|^2 ,
\]
which establishes \cref{eq:simple_bound_gamma}.
Finally, taking $\gamma_* \to \tfrac{1}{3L}\sqrt{\tfrac{5}{26}}$ gives $16/\gamma_*^2 \to 748.8L^2$, hence
\[
\|G(x_k)\|^2 \;\le\; \frac{780.8\,L^2}{(k+4)^2}\|x_0-x^\star\|^2.
\]

\end{proof}

\subsection{Analysis of the Bilinear Game}
\label{app:bilinear}

\textbf{Example.}
Let $f(x,y)=Lxy$, with operator $G(x,y)=(Ly,-Lx)$ and solution $x^\star=(0,0)$.
Fix $a,b>0$ with $b>a$, constant step size $\eta_\star=\tfrac{1}{\theta L}$
($\theta>0$), and anchoring coefficient $\beta_k=\tfrac{a}{k+b}$. For any
$(p,q)\in\mathbb{R}^2$, initialize~\eqref{eq:goma_triangle} at
$x_0=\tfrac{a\theta}{b-a}(-q,p)$. Then
\begin{equation}
  \|G(x_k)\|^2
  = \frac{(a\theta L)^2}{(k+b-a)^2}\,\|x_0-x^\star\|^2\,\bigl(1+o(1)\bigr).\label{eq:bilinear-rate}
\end{equation}

\begin{proof}
The bilinear operator satisfies the two elementary identities
\begin{equation}
  \|G(v)\|=L\,\|v\|,
  \qquad
  G^2=-L^2 I,
  \qquad v\in\mathbb{R}^2 ,
  \tag{P}\label{eq:bilinear-props}
\end{equation}
both immediate from $G(x,y)=(Ly,-Lx)$.

Multiplying the two lines of~\eqref{eq:goma_triangle} by $(k+b)$ and using
$(k+b)\beta_k=a$, $(k+b)(1-\beta_k)=k+b-a$, then setting the scaled variables
$\tilde x_k:=(k+b-a)x_k$ and $\tilde y_k:=(k+b)y_k$, linearity of $G$ gives the
exact system
\begin{align}
  \tilde y_k
  &= \tilde x_k + a x_0
     - \eta_\star\,\tfrac{k+b-a}{k+b-1}\,G(\tilde y_{k-1}),\label{eq:rescaled-a}\\
  \tfrac{k+b}{k+b-a+1}\,\tilde x_{k+1}
  &= \tilde x_k + a x_0 - \eta_\star\,G(\tilde y_k).\label{eq:rescaled-b}
\end{align}
Both prefactors tend to $1$ (and equal $1$ for all $k$ iff $a=1$), so the
limiting system is time-invariant with fixed point
$\tilde x_k=\tilde y_k\equiv\tilde x$ given by $\eta_\star G(\tilde x)=a x_0$.

Hence $G(\tilde x)=a\theta L\,x_0$, and taking norms via
$\|G(\tilde x)\|=L\|\tilde x\|$ from~\eqref{eq:bilinear-props},
\begin{equation}
  \|\tilde x\| = a\theta\,\|x_0\|.\label{eq:fp-magnitude}
\end{equation}
It remains to show $\tilde x_k\to\tilde x$. With $d_k:=\tilde x_k-\tilde x$ and
$A:=\eta_\star G$ (so $A^2=-g^2 I$, $g:=\eta_\star L=1/\theta$,
by~\eqref{eq:bilinear-props}), the deviation obeys $d_{k+1}=(I-2A)d_k+A\,d_{k-1}$ (subtracting the fixed
point from the limiting system gives $e_k=d_k-Ae_{k-1}$ and
$d_{k+1}=d_k-Ae_k$ with $e_k:=\tilde y_k-\tilde x$; eliminate $e_k$ using
$A^{-1}=-g^{-2}A$).. Since $\{I,A\}$ commute and $A^2=-g^2I$, the
characteristic roots have $|\lambda_\pm|^2=r_\pm:=\tfrac{1\pm\sqrt{1-4g^2}}{2}$.
Any admissible step $\eta_\star\le\tfrac{1}{2\sqrt3 L}$ gives $4g^2\le\tfrac13<1$,
hence $r_\pm\in(0,1)$ and $|\lambda_\pm|<1$: the homogeneous map is a
contraction. As the prefactors in~\eqref{eq:rescaled-a}--\eqref{eq:rescaled-b}
differ from $1$ by $\mathcal{O}(1/k)\to0$, it stays a contraction for large $k$, so
$d_k\to0$.

Finally, from $x_k=\tilde x_k/(k+b-a)$, $\tilde x_k\to\tilde x$,
$\|G(x_k)\|=L\|x_k\|$, and~\eqref{eq:fp-magnitude},
\[
  \|G(x_k)\|^2
  = \frac{L^2\|\tilde x_k\|^2}{(k+b-a)^2}
  \;\longrightarrow\;
  \frac{(a\theta L)^2}{(k+b-a)^2}\,\|x_0\|^2 ,
\]
which is~\eqref{eq:bilinear-rate} since $x^\star=0$.
\end{proof}

\paragraph{Comparison with the general guarantee.}
With $a=2$, $b=6$ and the largest admissible step $\eta_\star=\tfrac{1}{2\sqrt3 L}$
($\theta=2\sqrt3$), we have $k+b-a=k+4$ and $(a\theta L)^2=48L^2$, so
$\|G(x_k)\|^2=\tfrac{48L^2}{(k+4)^2}\|x_0-x^\star\|^2(1+o(1))$. The universal
bound~\eqref{eq:edge_bound1}, with $1/\eta_\star^2=\theta^2 L^2$, equals
$\tfrac{264L^2}{(k+6)^2}$. Both decay as $L^2/k^2$, and the ratio of leading
constants is independent of $k$:
\[
  \frac{16\theta^2+72}{(a\theta)^2}\bigg|_{a=2}
  = 4+\frac{18}{\theta^2}
  \;\xrightarrow{\ \theta=2\sqrt3\ }\;
  \frac{11}{2}= 5.5 .\]
So on this
hard instance the universal analysis is tight up to a constant factor at
most~$5.5$.

\section{Proof of \cref{sec:stochastic}}

\subsection{Setup}
\label{appx:stoch_setup}
 
Throughout this section, $G$ is monotone and $L$-Lipschitz with
$G(x^\star)=0$, and $\widehat G(x,\xi)$ satisfies
Assumption~\ref{ass:bounded_v}, where, as argued after
Assumption~\ref{ass:bounded_v}, we may take $\kappa\ge1$ without loss
of generality. We write $d_0:=\|x_0-x^\star\|$ and let
$\mathcal{F}_k:=\sigma(x_0,\xi_0,\dots,\xi_{k-1})$ denote the natural
filtration of the algorithm, so that $x_k,y_k\in\mathcal{F}_k$. We
consider the stochastic updates~\eqref{eq:sto_updates} with the
schedules
\begin{equation}
\label{eq:stoch_schedules}
\beta_k=\frac{1}{k+2},
\qquad
\eta_k=\frac{1}{L\sqrt{\kappa}\,(k+2)^{3/4}},
\end{equation}
for which
\begin{equation}
\label{eq:eta_facts}
L^2\eta_k^2=\frac{1}{\kappa(k+2)^{3/2}}\le\frac{1}{(k+2)^{3/2}},
\qquad
\sum_{k\ge0}L^2\eta_k^2
\le\frac1\kappa\sum_{m\ge2}m^{-3/2}\le\frac2\kappa\le2 ,
\end{equation}
together with the deterministic reference
trajectory~$(\bar x_k,\bar y_k)$ defined in~\eqref{eq:reference_traj}.
Since
$\|G(x_N)\|^2\le2\|G(\bar x_N)\|^2+2L^2\|x_N-\bar x_N\|^2$ by
$L$-Lipschitzness, Theorem~\ref{thm:last_iter_stoch_rho0} follows from
Lemma~\ref{lem:det-reference} (proved in
\cref{appx:lem_det_reference}) and Lemma~\ref{lem:stoch-stability}
(proved in \cref{appx:lem_stoch_stability}), both stated in
Section~\ref{sec:stochastic}.

 
\subsection{Proof of Lemma~\ref{lem:det-reference} (deterministic reference bound)}
\label{appx:lem_det_reference}

\begin{replemma}{lem:det-reference}
Let $G$ be monotone and $L$-Lipschitz with $G(x^\star)=0$, and let
$(\bar x_k,\bar y_k)$ be given by~\eqref{eq:reference_traj} with
$\beta_k=\frac{1}{k+2}$, $\eta_k=\frac{1}{L\sqrt{\kappa}(k+2)^{3/4}}$,
and $\kappa\ge1$. Then for all $N\ge1$ and all $k\ge0$,
\begin{align}
\|G(\bar x_N)\|^2
& \le 
\frac{33\,L^2\kappa\,\|x_0-x^\star\|^2}{\sqrt{N+1}},
\repeq{eq:lemA_xbound}\\
\|G(\bar y_k)\|^2
&\le \frac{94\,L^2\kappa\,\|x_0-x^\star\|^2}{\sqrt{k+1}}.
\repeq{eq:lemA_ybound}
\end{align}
\end{replemma}
 
\begin{proof}
\textbf{Step 1: Potential function and one-step decrease.}
Consider the two-term anchored potential function evaluated along the
reference trajectory, namely the analogue of the deterministic
potential~\eqref{eq:potential} in which the extrapolation term
$c_kL^2\|x_k-y_{k-1}\|^2$ is dropped, since the simplified variant
sets $\gamma_k=0$ and does not reuse past gradients:
\begin{equation}
\label{eq:det_potential}
H_k:=a_k\|G(\bar x_k)\|^2+b_k\langle G(\bar x_k),\bar x_k-x_0\rangle,
\qquad
a_{k+1}=\frac{b_k\eta_k}{2\beta_k(1-\beta_k)},
\qquad
b_{k+1}=\frac{b_k}{1-\beta_k}.
\end{equation}
The first term is the residual we control; the second is an anchoring
cross-term coupling the gradient with the displacement from the
reference point $x_0$.
For the schedules~\eqref{eq:stoch_schedules}, normalizing $b_0=1$,
\begin{equation}
\label{eq:ab_explicit}
b_k=k+1,
\qquad
a_k=\frac{(k+1)^{5/4}}{2L\sqrt{\kappa}}\quad(k\ge1),
\qquad\text{and}\qquad
b_1\eta_0=a_1 .
\end{equation}
From~\eqref{eq:reference_traj} we have the identities
\[
\bar x_{k+1}-\bar x_k=\beta_k(x_0-\bar x_k)-\eta_kG(\bar y_k),
\qquad
\bar x_{k+1}-\bar x_k
=\tfrac{\beta_k}{1-\beta_k}(x_0-\bar x_{k+1})
-\tfrac{\eta_k}{1-\beta_k}G(\bar y_k).
\]
Monotonicity of $G$ gives
\[
\langle G(\bar x_{k+1})-G(\bar x_k),\,\bar x_{k+1}-\bar x_k\rangle\ge0,
\qquad\text{i.e.}\qquad
\langle G(\bar x_{k+1}),\,\bar x_{k+1}-\bar x_k\rangle
\ge\langle G(\bar x_k),\,\bar x_{k+1}-\bar x_k\rangle .
\]
We substitute the second identity on the left-hand side and the first
identity on the right-hand side:
\[
\Big\langle G(\bar x_{k+1}),\,
\tfrac{\beta_k}{1-\beta_k}(x_0-\bar x_{k+1})
-\tfrac{\eta_k}{1-\beta_k}G(\bar y_k)\Big\rangle
\ge
\big\langle G(\bar x_k),\,\beta_k(x_0-\bar x_k)-\eta_kG(\bar y_k)\big\rangle .
\]
Expanding both inner products gives
\[
\tfrac{\beta_k}{1-\beta_k}\langle G(\bar x_{k+1}),x_0-\bar x_{k+1}\rangle
-\tfrac{\eta_k}{1-\beta_k}\langle G(\bar x_{k+1}),G(\bar y_k)\rangle
\ge
\beta_k\langle G(\bar x_k),x_0-\bar x_k\rangle
-\eta_k\langle G(\bar x_k),G(\bar y_k)\rangle .
\]
Moving the anchoring cross-terms (the $x_0-\bar x$ inner products) to the
left-hand side and the $G(\bar y_k)$ cross-terms to the right-hand side,
\[
\tfrac{\beta_k}{1-\beta_k}\langle G(\bar x_{k+1}),x_0-\bar x_{k+1}\rangle
-\beta_k\langle G(\bar x_k),x_0-\bar x_k\rangle
\ge
\tfrac{\eta_k}{1-\beta_k}\langle G(\bar x_{k+1}),G(\bar y_k)\rangle
-\eta_k\langle G(\bar x_k),G(\bar y_k)\rangle .
\]
Multiplying by $b_k/\beta_k$, using $b_{k+1}=\tfrac{b_k}{1-\beta_k}$ on the
$G(\bar x_{k+1})$ term and $\langle G(\bar x),\,x_0-\bar x\rangle
=-\langle G(\bar x),\,\bar x-x_0\rangle$ on both anchoring terms, yields
\begin{equation}
\label{eq:det_anchor_ineq}
b_k\langle G(\bar x_k),\bar x_k-x_0\rangle
-b_{k+1}\langle G(\bar x_{k+1}),\bar x_{k+1}-x_0\rangle
\ge
\tfrac{b_k\eta_k}{\beta_k(1-\beta_k)}
\langle G(\bar x_{k+1}),G(\bar y_k)\rangle
-\tfrac{b_k\eta_k}{\beta_k}\langle G(\bar x_k),G(\bar y_k)\rangle .
\end{equation}
By $L$-Lipschitzness and $\bar x_{k+1}-\bar y_k=-\eta_kG(\bar y_k)$,
\begin{equation}
\label{eq:det_lip}
\|G(\bar x_{k+1})\|^2
-2\langle G(\bar x_{k+1}),G(\bar y_k)\rangle
+\big(1-L^2\eta_k^2\big)\|G(\bar y_k)\|^2\le0 .
\end{equation}
Adding the $a_k$-terms to~\eqref{eq:det_anchor_ineq} and then adding
$a_{k+1}\times$\eqref{eq:det_lip}, the $\|G(\bar x_{k+1})\|^2$ terms
and the $\langle G(\bar x_{k+1}),G(\bar y_k)\rangle$ terms cancel
(using $\tfrac{b_k\eta_k}{\beta_k(1-\beta_k)}=2a_{k+1}$), leaving the
one-step inequality
\begin{equation}
\label{eq:det_onestep}
H_k-H_{k+1}
\;\ge\;
a_k\|G(\bar x_k)\|^2
-2a_{k+1}(1-\beta_k)\langle G(\bar x_k),G(\bar y_k)\rangle
+a_{k+1}\big(1-L^2\eta_k^2\big)\|G(\bar y_k)\|^2 .
\end{equation}
The right-hand side of~\eqref{eq:det_onestep} is a quadratic form in
$\big(\|G(\bar x_k)\|,\|G(\bar y_k)\|\big)$ and is nonnegative
provided
\begin{equation}
\label{eq:psd_condition}
a_ka_{k+1}\big(1-L^2\eta_k^2\big)\ge a_{k+1}^2(1-\beta_k)^2,
\qquad\text{i.e.}\qquad
1-L^2\eta_k^2\ge\Big(\frac{k+1}{k+2}\Big)^{3/4},
\end{equation}
using $a_{k+1}/a_k=\big((k+2)/(k+1)\big)^{5/4}$ from
\eqref{eq:ab_explicit}. By the tangent bound
$\big(1-\frac{1}{k+2}\big)^{3/4}\le1-\frac{3}{4(k+2)}$ (concavity of
$t\mapsto t^{3/4}$) and $L^2\eta_k^2\le(k+2)^{-3/2}$ from
\eqref{eq:eta_facts}, condition~\eqref{eq:psd_condition} holds as soon
as $(k+2)^{-3/2}\le\frac{3}{4(k+2)}$, i.e.\
$(k+2)^{1/2}\ge\frac43$, which is true for every $k\ge0$. Hence
$H_{k+1}\le H_k$ for all $k\ge1$, so
\begin{equation}
\label{eq:HN_le_H1}
H_N\le H_1\qquad(N\ge1).
\end{equation}
(Starting the telescoping at $k=1$ sidesteps the $a_0=0$ boundary
step entirely.)
 
\textbf{Step 2: lower bound on $H_N$.}
Since $G(x^\star)=0$, monotonicity gives
$\langle G(\bar x_N),\bar x_N-x^\star\rangle\ge0$, so writing
$\bar x_N-x_0=(\bar x_N-x^\star)+(x^\star-x_0)$ and applying Cauchy-Schwarz and then Young's
inequalities,
\begin{equation}
\label{eq:HN_lower}
H_N
\ge a_N\|G(\bar x_N)\|^2-b_N\|G(\bar x_N)\|\,d_0
\ge\frac{a_N}{2}\|G(\bar x_N)\|^2-\frac{b_N^2}{2a_N}d_0^2 .
\end{equation}
Combining~\eqref{eq:HN_le_H1} and~\eqref{eq:HN_lower},
\begin{equation}
\label{eq:GxN_master}
\|G(\bar x_N)\|^2
\le\frac{2H_1}{a_N}+\Big(\frac{b_N}{a_N}\Big)^2d_0^2,
\qquad
\Big(\frac{b_N}{a_N}\Big)^2=\frac{4L^2\kappa}{\sqrt{N+1}} .
\end{equation}
 
\textbf{Step 3: the term $H_1/a_N$ is of lower order.}
Since $\beta_0=\frac12$ and $\bar x_0=x_0$, we have $\bar y_0=x_0$ and
$\bar x_1=x_0-\eta_0G(x_0)$. As $L\eta_0\le1$
by~\eqref{eq:eta_facts}, Lipschitzness gives
$\|G(\bar x_1)\|\le(1+L\eta_0)\|G(x_0)\|\le2Ld_0$, and
\[
\langle G(\bar x_1),\bar x_1-x_0\rangle
=-\eta_0\langle G(\bar x_1),G(x_0)\rangle
\le\eta_0\|G(\bar x_1)\|\,\|G(x_0)\|
\le2\eta_0L^2d_0^2 .
\]
With $b_1\eta_0=a_1$ this yields
$H_1\le4a_1L^2d_0^2+2b_1\eta_0L^2d_0^2=6a_1L^2d_0^2$, and since
$a_1/a_N=\big(2/(N+1)\big)^{5/4}$,
\begin{equation}
\label{eq:H1_over_aN}
\frac{2H_1}{a_N}
\le\frac{12\cdot2^{5/4}L^2d_0^2}{(N+1)^{5/4}}\,,
\end{equation}
 Plugging
\eqref{eq:H1_over_aN} into~\eqref{eq:GxN_master} with $\kappa \geq 1$ proves
\eqref{eq:lemA_xbound}.
 
\textbf{Step 4: boundedness of the reference iterates.}
Since $\bar x_{k+1}=\bar y_k-\eta_kG(\bar y_k)$, monotonicity and
Lipschitzness (with $G(x^\star)=0$) give
\[
\|\bar x_{k+1}-x^\star\|^2
=\|\bar y_k-x^\star\|^2
-2\eta_k\langle G(\bar y_k),\bar y_k-x^\star\rangle
+\eta_k^2\|G(\bar y_k)\|^2
\le\big(1+L^2\eta_k^2\big)\|\bar y_k-x^\star\|^2 .
\]
Together with
$\|\bar y_k-x^\star\|\le\beta_kd_0+(1-\beta_k)\|\bar x_k-x^\star\|$,
this convex combination keeps $\|\bar y_k-x^\star\|$ no larger than
$\max\big(d_0,\|\bar x_k-x^\star\|\big)$, so each iteration inflates the
distance to $x^\star$ by at most the factor $(1+L^2\eta_k^2)^{1/2}$.
Accumulating these factors from $\bar x_0=x_0$, and using $1+t\le e^t$
with $\sum_{j\ge0}L^2\eta_j^2\le2$ from~\eqref{eq:eta_facts}, gives, for
all $k$,
\begin{equation}
\label{eq:bounded_iterates}
\|\bar x_k-x^\star\|
\le\prod_{j\ge0}\big(1+L^2\eta_j^2\big)^{1/2}d_0
\le e^{\frac12\sum_jL^2\eta_j^2}d_0\le e\,d_0,
\qquad
\|\bar y_k-x^\star\|\le e\,d_0,
\end{equation}
and consequently
\begin{equation}
\label{eq:y_minus_x}
\|\bar y_k-\bar x_k\|=\beta_k\|x_0-\bar x_k\|
\le(1+e)\beta_k d_0 .
\end{equation}
 
\textbf{Step 5: residual bound at $\bar y_k$.}
For $k\ge1$, Lipschitzness, the bound~\eqref{eq:lemA_xbound} applied
with $N=k$, and~\eqref{eq:y_minus_x} give
\[
\|G(\bar y_k)\|^2
\le2\|G(\bar x_k)\|^2+2L^2\|\bar y_k-\bar x_k\|^2
\le\frac{66L^2\kappa d_0^2}{\sqrt{k+1}}
+\frac{2(1+e)^2L^2d_0^2}{(k+2)^2}
\le\frac{\big(66+2(1+e)^2\big)L^2\kappa d_0^2}
{\sqrt{k+1}},
\]
since $(k+2)^2\ge\sqrt{k+1}$ and $\kappa\ge1$. For $k=0$,
$\bar y_0=x_0$ and $\|G(x_0)\|^2\le L^2d_0^2$. This proves
\eqref{eq:lemA_ybound} with
$66+2(1+e)^2\le94$.
\end{proof}
 
\subsection{Proof of Lemma~\ref{lem:stoch-stability} (stochastic stability)}
\label{appx:lem_stoch_stability}

\begin{replemma}{lem:stoch-stability}
In the setting of Lemma~\ref{lem:det-reference}, let $\widehat G$
satisfy Assumption~\ref{ass:bounded_v}, let $(x_k)$ be given
by~\eqref{eq:sto_updates}, and set $e_k:=x_k-\bar x_k$. Then for all $N\ge0$,
\begin{equation}\repeq{eq:lemB_precise}
\mathbb{E}\,\|e_N\|^2
\;\le\;
\frac{1}{\sqrt{N+1}}
\left(\frac{4\,\sigma^2}{L^2\kappa}
+ 752\,\kappa\,\|x_0-x^\star\|^2\right).
\end{equation}
\end{replemma}

\begin{proof}
\textbf{Step 1: error recursion.}
Let $n_k:=\widehat G(y_k,\xi_k)-G(y_k)$, so
$\mathbb{E}[n_k\mid\mathcal{F}_k]=0$ and, by
Assumption~\ref{ass:bounded_v} and unbiasedness,
\begin{equation}
\label{eq:noise_variance}
\mathbb{E}\big[\|n_k\|^2\mid\mathcal{F}_k\big]
=\mathbb{E}\big[\|\widehat G(y_k,\xi_k)\|^2\mid\mathcal{F}_k\big]
-\|G(y_k)\|^2
\le\sigma^2+\kappa\|G(y_k)\|^2 .
\end{equation}
Subtracting the reference update~\eqref{eq:reference_traj} from the
stochastic update~\eqref{eq:sto_updates}, and using
$y_k-\bar y_k=(1-\beta_k)e_k$,
\begin{equation}
\label{eq:error_recursion}
e_{k+1}
=(1-\beta_k)e_k-\eta_k\big(G(y_k)-G(\bar y_k)\big)-\eta_kn_k .
\end{equation}
The first two terms on the right-hand side are
$\mathcal{F}_k$-measurable, so taking conditional expectations and
using $\mathbb{E}[n_k\mid\mathcal{F}_k]=0$,
\begin{equation}
\label{eq:cond_expansion}
\mathbb{E}\big[\|e_{k+1}\|^2\mid\mathcal{F}_k\big]
=\big\|(1-\beta_k)e_k-\eta_k\big(G(y_k)-G(\bar y_k)\big)\big\|^2
+\eta_k^2\,\mathbb{E}\big[\|n_k\|^2\mid\mathcal{F}_k\big] .
\end{equation}
 
\textbf{Step 2: the drift term is contractive.}
By monotonicity,
$\langle G(y_k)-G(\bar y_k),y_k-\bar y_k\rangle\ge0$, and since
$y_k-\bar y_k=(1-\beta_k)e_k$ with $\beta_k<1$,
\begin{equation}
\label{eq:mono_cross}
\langle G(y_k)-G(\bar y_k),e_k\rangle\ge0 .
\end{equation}
By Lipschitzness,
$\|G(y_k)-G(\bar y_k)\|\le L(1-\beta_k)\|e_k\|$. Expanding the square
and using~\eqref{eq:mono_cross} to drop the cross term,
\begin{equation}
\label{eq:drift_contraction}
\big\|(1-\beta_k)e_k-\eta_k\big(G(y_k)-G(\bar y_k)\big)\big\|^2
\le(1-\beta_k)^2\big(1+L^2\eta_k^2\big)\|e_k\|^2 .
\end{equation}
Moreover, again by Lipschitzness,
\begin{equation}
\label{eq:Gyk_split}
\|G(y_k)\|^2\le2\|G(\bar y_k)\|^2+2L^2(1-\beta_k)^2\|e_k\|^2 .
\end{equation}
 
\textbf{Step 3: one-step inequality.}
Combining \eqref{eq:cond_expansion}--\eqref{eq:Gyk_split} with
\eqref{eq:noise_variance},
\begin{equation}
\label{eq:one_step_stab}
\mathbb{E}\big[\|e_{k+1}\|^2\mid\mathcal{F}_k\big]
\le q_k\|e_k\|^2+\eta_k^2\sigma^2+2\kappa\eta_k^2\|G(\bar y_k)\|^2,
\qquad
q_k:=(1-\beta_k)^2\Big(1+(1+2\kappa)L^2\eta_k^2\Big) .
\end{equation}
Since $\kappa\ge1$ implies $(1+2\kappa)/\kappa\le3$, the schedule
\eqref{eq:stoch_schedules} gives
$(1+2\kappa)L^2\eta_k^2\le3(k+2)^{-3/2}$, hence, with $m:=k+2\ge2$,
\begin{equation}
\label{eq:qk_bound}
q_k\le\Big(1-\frac1m\Big)^2\Big(1+\frac{3}{m^{3/2}}\Big)
\le1-\frac{3}{4m} .
\end{equation}
The second inequality in~\eqref{eq:qk_bound} is elementary: it is
equivalent to $g(m):=3m^{-1/2}(1-1/m)^2+1/m\le\frac54$, which holds
for $m\ge9$ since then $3m^{-1/2}\le1$ and $1/m\le\frac19$, and is
verified directly for $m\in\{2,\dots,8\}$ (its maximum there is
$g(3)\le1.11$).
 
\textbf{Step 4: solving the recursion.}
Let $r_k:=\mathbb{E}\|e_k\|^2$, so $r_0=0$. Taking total expectations
in~\eqref{eq:one_step_stab}, with
$\eta_k^2\sigma^2=\sigma^2/\big(L^2\kappa(k+2)^{3/2}\big)$ and, by
\eqref{eq:lemA_ybound} and $\sqrt{k+1}\ge1$,
\[
2\kappa\eta_k^2\|G(\bar y_k)\|^2
\le\frac{2\kappa}{L^2\kappa(k+2)^{3/2}}\cdot
\frac{94L^2\kappa d_0^2}{\sqrt{k+1}}
\le\frac{188\kappa d_0^2}{(k+2)^{3/2}},
\]
we obtain
\begin{equation}
\label{eq:r_recursion}
r_{k+1}\le\Big(1-\frac{3}{4(k+2)}\Big)r_k+\frac{S}{(k+2)^{3/2}},
\qquad
S:=\frac{\sigma^2}{L^2\kappa}+188\kappa d_0^2 .
\end{equation}
We claim that~\eqref{eq:r_recursion} with $r_0=0$ implies
$r_k\le4S/\sqrt{k+2}$ for all $k\ge0$. The base case is immediate; if
$r_k\le4S\,m^{-1/2}$ with $m=k+2$, then
\[
r_{k+1}
\le\Big(1-\frac{3}{4m}\Big)\frac{4S}{\sqrt m}+\frac{S}{m^{3/2}}
=\frac{4S}{\sqrt m}-\frac{2S}{m^{3/2}}
\le\frac{4S}{\sqrt{m+1}},
\]
where the last step uses
\[
\frac{1}{\sqrt m}-\frac{1}{\sqrt{m+1}}
=\frac{1}{\sqrt m\,\sqrt{m+1}\,(\sqrt m+\sqrt{m+1})}
\le\frac{1}{2m^{3/2}} .
\]
This proves
$r_N\le\frac{4S}{\sqrt{N+1}}$, which is exactly
\eqref{eq:lemB_precise}.
\end{proof}
 
\subsection{Proof of Theorem~\ref{thm:last_iter_stoch_rho0}}
\label{appx:last_iter_stoch_rho0}
\begin{reptheorem}{thm:last_iter_stoch_rho0}
Let \(G:\mathbb{R}^d\to\mathbb{R}^d\) be monotone and \(L\)-Lipschitz and $\widehat G(x,\xi)$ be a stochastic oracle following Assumption \ref{ass:bounded_v}. Then for the updates described in~\eqref{eq:sto_updates} with $\beta_k=\frac{1}{k+2}$ and $\eta_k=\frac{1}{L\sqrt{\kappa} (k+2)^{3/4}}$, we have for all $N\ge 0$,
\begin{align*}
\mathbb{E}\|G(x_N)\|^2 \;\le\; \frac{ 1570 L^2 \kappa \|x_0-x^\star\|^2}{\sqrt{N+1}} + \frac{8\,\sigma^2 }{\kappa\sqrt{N+1}}.
\end{align*}    
\end{reptheorem}
 
\begin{proof}[Proof of Theorem~\ref{thm:last_iter_stoch_rho0}]
For $N=0$ the bound holds trivially:
$\|G(x_0)\|^2\le L^2d_0^2\le 1570L^2\kappa d_0^2$ since $\kappa\ge1$. Let $N\ge1$. By $L$-Lipschitzness of $G$,
\[
\|G(x_N)\|^2\le2\|G(\bar x_N)\|^2+2L^2\|x_N-\bar x_N\|^2 .
\]
Taking expectations and applying Lemma~\ref{lem:det-reference} and
Lemma~\ref{lem:stoch-stability},
\[
\mathbb{E}\,\|G(x_N)\|^2
\le\frac{66L^2\kappa d_0^2}{\sqrt{N+1}}
+\frac{2L^2}{\sqrt{N+1}}
\left(\frac{4\sigma^2}{L^2\kappa}+752\kappa d_0^2\right)
=\frac{1570L^2\kappa d_0^2}{\sqrt{N+1}}
+\frac{8\sigma^2}{\kappa\sqrt{N+1}} .
\]
\end{proof}

\subsection{Proof of Non-stochastic Rate of \eqref{eq:goma_gamma0}}

\begin{theorem}\label{thm:goma_gamma0_det}
Assume $G$ is monotone and $L$-Lipschitz. Consider
\[
y_k=\beta_k x_0+(1-\beta_k)x_k,\qquad
x_{k+1}=y_k-\eta_k\,G(y_k),
\]
with $(y_{-1}=x_0)$.
Potential of \eqref{eq:potential} with the coefficients $a_{k+1} = \frac{b_k\eta_k}{2\beta_k(1 - \beta_k)}$ and $b_{k+1}=\tfrac{b_k}{1-\beta_k}$, $c_k=a_k$, and the choice of parameters:
$$\eta_k = \frac{c}{L}\sqrt{\frac{\beta_k}{2}}\,,\quad c \in (0, \frac{1}{\sqrt{2}}]\,,\quad \beta_k = \frac{1}{k+2},$$

admits, with $\widetilde c_{k+1}\ge0$ for $\theta=1$, the one-step decrease $V_k-V_{k+1}\ge \widetilde c_{k+1}L^2\|x_{k+1}-y_k\|^2$, hence $V_{k+1}\le V_k$ for all $k\ge1$. Consequently, for all $k\ge1$,
\[
\|G(x_k)\|^2
\;\le\;
\frac{C_{\mathrm{init}}\,L^2}{(k+1)^{3/2}}\,\|x_0-x^\star\|^2
\;+\;
\frac{C_\star\,L^2}{k+1}\|x_0-x^\star\|^2. 
\]

In particular, $\|G(x_k)\|^2=\mathcal{O}(1/(k+1))$.

\end{theorem}

\begin{proof}
From the update rules we obtain:
\begin{align}
x_{k+1}-x_k &= \beta_k(x_0-x_k) - \eta_k G(y_k), \label{eq:h1_goma}\\
x_{k+1}-x_k &= \frac{\beta_k}{1-\beta_k}(x_0 - x_{k+1}) - \frac{\eta_k}{1-\beta_k}G(y_k), \label{eq:h2_goma}\\
x_{k+1}-y_k &= -\eta_k G(y_k).\label{eq:h3_goma}
\end{align}

By monotonicity of $G$:
$$
  \langle G(x_{k+1})-G(x_k),\,x_{k+1}-x_k \rangle \ge 0,
$$
$$
  \langle G(x_{k+1}),\,x_{k+1}-x_k \rangle \ge \langle G(x_k),\, x_{k+1}-x_k\rangle.
$$
Using \cref{eq:h1_goma} and \cref{eq:h2_goma}, we write:
$$
  \langle G(x_{k+1}),\, \frac{\beta_k}{1-\beta_k}(x_0 - x_{k+1}) - \frac{\eta_k}{1-\beta_k}G(y_k) \rangle \ge \langle G(x_k),\,\beta_k(x_0-x_k) - \eta_k G(y_k) \rangle.
$$
Rearranging:
$$
  \frac{\beta_k}{1-\beta_k}\langle G(x_{k+1}),\, x_0 - x_{k+1} \rangle  \ge \beta_k\langle G(x_k),\,x_0-x_k \rangle - \eta_k\langle G(x_k),\, G(y_k) \rangle +  \frac{\eta_k}{1-\beta_k}\langle G(x_{k+1}),\, G(y_k)\rangle.
$$
Multiplying this inequality by $\frac{b_k}{\beta_k}$ and taking $b_{k+1} = \frac{b_k}{1 - \beta_k}$:
\begin{align}
    b_k\langle G(x_k),x_k-x_0\rangle - b_{k+1}\langle G(x_{k+1}),x_{k+1}-x_0\rangle
    \ge &
    \frac{b_k\,\eta_k}{\beta_k(1-\beta_k)}\langle G(x_{k+1}),G(y_k)\rangle
    - \frac{b_k\,\eta_k}{\beta_k}\langle G(x_k),G(y_k)\rangle. \label{eq:monotone_final}
\end{align}
Recall:
\begin{equation*}
     V_k = a_k\|G(x_k)\|^2 + b_k\langle G(x_k),\,x_k - x_0\rangle+ c_kL^2\|x_k - y_{k-1}\|^2.
\end{equation*}
Adding $a_k$ and $c_k$ terms to \cref{eq:monotone_final} gives
\begin{align}
V_k - V_{k+1} &\ge a_k\|G(x_k)\|^2 - a_{k+1}\|G(x_{k+1})\|^2
+ \frac{b_k\,\eta_k}{\beta_k(1-\beta_k)}\langle G(x_{k+1}),G(y_k)\rangle
    - \frac{b_k\,\eta_k}{\beta_k}\langle G(x_k),G(y_k)\rangle \notag\\
&\quad
+ c_kL^2\|x_k-y_{k-1}\|^2 - c_{k+1}L^2\|x_{k+1}-y_k\|^2. \label{eq:v_diff_goma}
\end{align} 
Using Lipschitz continuity and \cref{eq:h3_goma}, we have
\begin{align}
\|G(x_{k+1})-G(y_k)\|^2 + \theta L^2\|x_{k+1}-y_k\|^2
&\le (1+\theta)L^2\|x_{k+1}-y_k\|^2 \notag\\
&= (1+\theta)L^2\|-\eta_k G(y_k)\|^2 .\label{eq:split_gamma_xk}
\end{align}
Expanding the left-hand side
\[
\|G(x_{k+1})-G(y_k)\|^2 + \theta L^2\|x_{k+1}-y_k\|^2
= \|G(x_{k+1})\|^2 - 2\langle G(x_{k+1}),G(y_k)\rangle + \|G(y_k)\|^2
+ \theta L^2\|x_{k+1}-y_k\|^2,
\]
and setting \(M:=L^2(1+\theta)\), inequality \eqref{eq:split_gamma_xk} is equivalent to
\begin{align}
&\|G(x_{k+1})\|^2
- 2\langle G(x_{k+1}),G(y_k)\rangle
+ \bigl(1 - M\eta_k^2\bigr)\|G(y_k)\|^2 
+ \theta L^2\|x_{k+1}-y_k\|^2
\;\le\;0. \label{eq:ineq_cross_form_gamma0}
\end{align}
Recall 
\begin{align}
V_k - V_{k+1}
\;\ge\;&
a_k\|G(x_k)\|^2 - a_{k+1}\|G(x_{k+1})\|^2
+ \frac{b_k\,\eta_k}{\beta_k(1-\beta_k)}\langle G(x_{k+1}),G(y_k)\rangle \notag\\
&\quad
- \frac{b_k\,\eta_k}{\beta_k}\langle G(x_k),G(y_k)\rangle + c_kL^2\|x_k-y_{k-1}\|^2 - c_{k+1}L^2\|x_{k+1}-y_k\|^2.\label{eq:Vdiff_gamma0_base_again}
\end{align}
Multiply \eqref{eq:ineq_cross_form_gamma0} by $\tfrac{b_k\eta_k}{2\beta_k(1-\beta_k)}\;>\;0$
and add the result to the right-hand side of \eqref{eq:Vdiff_gamma0_base_again} to obtain:
\begin{align}
V_k - V_{k+1}
\;\ge\;&
a_k\|G(x_k)\|^2 \notag\\
&\quad
+ \Bigl(\tfrac{b_k\eta_k}{2\beta_k(1-\beta_k)} -a_{k+1}\Bigr)\|G(x_{k+1})\|^2 \notag\\
&\quad
-\tfrac{b_k\,\eta_k}{\beta_k}\,\langle G(x_k),G(y_k)\rangle \notag\\
&\quad
+\tfrac{b_k\,\eta_k}{2\beta_k(1-\beta_k)}\bigl(1-M\eta_k^2\bigr)\|G(y_k)\|^2 \notag\\
&\quad
+L^2\Bigl[c_k\|x_k-y_{k-1}\|^2
+\Bigl(\tfrac{\theta\,b_k\,\eta_k}{2\beta_k(1-\beta_k)} - c_{k+1}\Bigr)\|x_{k+1}-y_k\|^2\Bigr].
\end{align}
We set $a_{k+1} = \frac{b_k\eta_k}{2\beta_k(1 - \beta_k)}$, implying that $\Bigl(\tfrac{b_k\eta_k}{2\beta_k(1-\beta_k)} -a_{k+1}\Bigr)\|G(x_{k+1})\|^2 = 0$. Also, we have $b_{k+1} = \frac{b_k}{1-\beta_k}$. Define the \(2\times2\) quadratic form of (\(\|G(x_{k})\|,\|G(y_k)\|\)) with the following matrix:
\[
S_k \;=\;
\begin{pmatrix}
S_k^{11} & S_k^{12}\\
S_k^{12} & S_k^{22}
\end{pmatrix},
\]
\begin{align*}
S_k^{11}
=a_k,\qquad 
S_k^{12} 
=-\tfrac{b_k\,\eta_k}{2\beta_k},\qquad
S_k^{22}
=\tfrac{b_k\,\eta_k}{2\beta_k(1-\beta_k)}\bigl(1-M\eta_k^2\bigr).
\end{align*}

Thus, sufficient conditions for \(V_k-V_{k+1}\ge0\) are:

\begin{align*}
&\text{(i)}\quad S_k\succeq0
\;\;\Longleftrightarrow\;\;
S_k^{11}\ge0,\; S_k^{22}\ge0,\;S_k^{11}  S_k^{22}\ge (S_k^{12})^2,\\
&\text{(ii)}\quad \frac{\theta\,b_k\,\eta_k}{2\beta_k(1-\beta_k)} \;\ge\; c_{k+1} \;\ge\; 0.
\end{align*}

We set $\eta_k = c\sqrt{\frac{\beta_k}{M}}$ with $\beta_k = \frac{1}{k+2}$. We choose and check if all conditions work:

\[
S_k^{11} = a_k,\qquad 
S_k^{12} = -\frac{b_k c}{2\sqrt{M\,\beta_k}},\qquad 
S_k^{22} = \frac{b_k c}{2\sqrt{M\,\beta_k}(1-\beta_k)}\bigl(1-c^2\beta_k\bigr).
\]

\paragraph{Condition (i)}
We take $0<c\le\tfrac{1}{\sqrt2}$, so that $c^2\le\tfrac12$. Since
$\beta_k=\tfrac{1}{k+2}\le\tfrac12$, this gives $1-c^2\beta_k\ge1-\tfrac14>0$,
hence $S_k^{22}>0$ (and $S_k^{11}=a_k\ge0$). It remains to check the
determinant condition $S_k^{11}S_k^{22}\ge(S_k^{12})^2$, that is,
\[
a_k \;\ge\; \frac{(S_k^{12})^2}{S_k^{22}}
= \frac{b_k c}{2\sqrt{M\,\beta_k}}\cdot\frac{1-\beta_k}{1-c^2\beta_k}.
\]
Using $a_{k}=\frac{b_{k-1}c}{2\sqrt{M\beta_{k-1}}(1-\beta_{k-1})}$, canceling
$c,M>0$, and using $\frac{b_k}{b_{k-1}}=\frac{1}{1-\beta_{k-1}}$, this is
equivalent to
\[
\sqrt{\frac{\beta_k}{\beta_{k-1}}} \;\ge\;
\frac{1-\beta_k}{1-c^2\beta_k}.
\tag{*}
\]
Since $c^2\le\tfrac12$ we have $1-c^2\beta_k\ge1-\tfrac12\beta_k$, so it
suffices to prove $\sqrt{\beta_k/\beta_{k-1}}\ge(1-\beta_k)/(1-\tfrac12\beta_k)$.
Writing $x:=k+2$, so that $\beta_k=\tfrac1x$ and $\beta_{k-1}=\tfrac{1}{x-1}$,
this reads
\[
\sqrt{\frac{x-1}{x}}
\;\ge\;
\frac{x-1}{\,x-\tfrac12\,}.
\tag{$\ast\ast$}
\]
Squaring and cross-multiplying (all terms are positive), then cancelling
$x-1>0$, reduces $(\ast\ast)$ to
\[
\bigl(x-\tfrac12\bigr)^2 \;\ge\; x(x-1),
\qquad\text{i.e.}\qquad
\tfrac14\ge0,
\]
which always holds. This verifies the determinant condition $S_k^{11}S_k^{22}\ge(S_k^{12})^2$ for every $k\ge1$, where $a_k=\frac{b_{k-1}c}{2\sqrt{M\beta_{k-1}}(1-\beta_{k-1})}>0$. The boundary step $k=0$ is excluded: there $a_0=0$, so $S_0^{11}=0$ while $S_0^{12}\ne0$, and the determinant condition cannot hold. (Note the formula for $a_0$ would require $\beta_{-1}=1$, i.e.\ division by zero, so $a_0$ is not fixed by the recursion.) We handle $k=0$ separately by telescoping from $k=1$.

\paragraph{Condition (ii)}
If we set $c_k = a_k$, then using
  \[
  a_{k+1} = \frac{b_k\eta_k}{2\beta_k(1-\beta_k)},
  \]
and choosing $\theta = 1$, the condition (ii) will be satisfied.

\medskip

Let \(H_k:=a_k\|G(x_k)\|^2+b_k\langle G(x_k),x_k-x_0\rangle\).
We rewrite \(x_k-x_0=(x_k-x^\star)+(x^\star-x_0)\). By monotonicity of \(G\) and the fact that \(G(x^\star)=0\),
\[
\langle G(x_k),x_k-x^\star\rangle \ge 0.
\]
Hence,
\[
\begin{aligned}
H_k
&\ge a_k\|G(x_k)\|^2 + b_k\langle G(x_k),x^\star-x_0\rangle \\
&\ge a_k\|G(x_k)\|^2 - \frac{a_k}{2}\|G(x_k)\|^2 - \frac{b_k^2}{2a_k}\|x_0-x^\star\|^2 \\
&= \frac{a_k}{2}\|G(x_k)\|^2 - \frac{b_k^2}{2a_k}\|x_0-x^\star\|^2,
\end{aligned}
\]
where the second inequality follows from Young’s inequality. Therefore,
\[
\frac{a_k}{2}\|G(x_k)\|^2 \;\le\; H_k + \frac{b_k^2}{2a_k}\|x_0-x^\star\|^2
\;\le\; V_k + \frac{b_k^2}{2a_k}\|x_0-x^\star\|^2.
\]
The PSD condition~(i) requires $S_k^{11}=a_k>0$, which holds for every $k\ge1$ but fails at $k=0$ (where $a_0=0$). We therefore telescope the one-step decrease from $k=1$: $V_{k+1}\le V_k$ for all $k\ge1$, so $V_k\le V_1$ for all $k\ge1$. (Starting the telescoping at $k=1$ sidesteps the $a_0=0$ boundary step entirely, exactly as in the proof of \cref{lem:det-reference}.) Combining $V_k\le V_1$ with $\tfrac{a_k}{2}\|G(x_k)\|^2\le V_k+\tfrac{b_k^2}{2a_k}\|x_0-x^\star\|^2$ and dividing by $a_k/2$ gives, for all $k\ge1$,
\begin{equation}
    \|G(x_k)\|^2 \;\le\; \frac{2V_1}{a_k} + 2\Bigl(\frac{b_k}{a_k}\Bigr)^2\|x_0-x^\star\|^2. \label{eq:theorem_helper}
\end{equation}

With \(\beta_k=\tfrac{1}{k+2}\), we have \(1-\beta_k=\tfrac{k+1}{k+2}\), and the recursion $b_{k+1}=b_k/(1-\beta_k)$ telescopes to
\[
b_k=b_0\prod_{t=0}^{k-1}\frac{1}{1-\beta_t}=b_0\prod_{t=0}^{k-1}\frac{t+2}{t+1}=b_0\,(k+1).
\]
Substituting into $a_{k+1}=\tfrac{b_k\eta_k}{2\beta_k(1-\beta_k)}$ with $\eta_k=\tfrac{c}{L}\sqrt{\tfrac{\beta_k}{2}}$ yields
\[
a_{k+1}=\frac{b_0c}{2\sqrt2\,L}\,(k+2)^{3/2},
\qquad\text{i.e.}\qquad
a_k=\frac{b_0c}{2\sqrt2\,L}\,(k+1)^{3/2}\quad(k\ge1).
\]
Hence
\[
2\Bigl(\frac{b_k}{a_k}\Bigr)^2=\frac{16L^2}{c^2\,(k+1)},
\qquad
\frac{a_1}{a_k}=\frac{2\sqrt2}{(k+1)^{3/2}}.
\]

It remains to bound $V_1$. Since $\beta_0=\tfrac12$ and $y_0=x_0$, we have $x_1=x_0-\eta_0G(x_0)$ with $L\eta_0=\tfrac{c}{2}\le1$, so Lipschitzness gives $\|G(x_1)\|\le(1+\tfrac{c}{2})\|G(x_0)\|$ and $\|x_1-y_0\|=\eta_0\|G(x_0)\|$. Using $c_1=a_1$, the identity $b_1\eta_0=a_1$, $L^2\eta_0^2=\tfrac{c^2}{4}$, and $\|G(x_0)\|\le L\|x_0-x^\star\|$,
\[
V_1=a_1\|G(x_1)\|^2+b_1\langle G(x_1),x_1-x_0\rangle+a_1L^2\|x_1-x_0\|^2
\le a_1\Bigl[(1+\tfrac c2)^2+(1+\tfrac c2)+\tfrac{c^2}{4}\Bigr]L^2\|x_0-x^\star\|^2
\le 4\,a_1L^2\|x_0-x^\star\|^2,
\]
the last step using $c\le1$. Therefore
\[
\frac{2V_1}{a_k}\le 8L^2\,\frac{a_1}{a_k}\,\|x_0-x^\star\|^2
=\frac{16\sqrt2\,L^2}{(k+1)^{3/2}}\|x_0-x^\star\|^2,
\]
and substituting into \eqref{eq:theorem_helper},
\[
\|G(x_k)\|^2
\;\le\;
\frac{16\sqrt2\,L^2}{(k+1)^{3/2}}\,\|x_0-x^\star\|^2
\;+\;
\frac{16\,L^2}{c^2\,(k+1)}\|x_0-x^\star\|^2,
\]
which is the claimed bound with \(C_{\mathrm{init}}=16\sqrt2\) and \(C_\star=16/c^2\). Taking the largest step $c=\tfrac{1}{\sqrt2}$ gives
\[
\|G(x_k)\|^2
\;\le\;
\frac{16\sqrt2\,L^2}{(k+1)^{3/2}}\|x_0-x^\star\|^2
+\frac{32\,L^2}{k+1}\|x_0-x^\star\|^2,
\]
which proves the rate \(\|G(x_k)\|^2=\mathcal{O}((k+1)^{-1}) \).

\end{proof}

\section{Experimental Detail}
\subsection{Negative comonotonicity example ~\S\ref{sec:quadratic}}
we define
\emph{$\rho$-Comonotonicity},  for some $\rho\in\!\left(-\tfrac{1}{2L},\infty\right)$ the field $F$ satisfies
\[
\langle Fz-Fz',\, z-z'\rangle \;\ge\; \rho\,\|Fz-Fz'\|^2,\qquad \forall\, z,z'\in\mathbb{R}^d .
\]\label{ass:rho-comon}
where L is the lipschitz constant. $\rho<0$ is referred to as \emph{negative comonotonicity}.\\

We demonstrate the behavior of \eqref{eq:goma} with  $\eta_k = 0.2$, $\gamma_k = 0.8(1 - \beta_k)$, and $\beta_k = \frac{2}{k+6}$ on the following two–player quadratic saddle game with $L$-Lipschitz and \hyperref[ass:rho-comon]{$\rho$-comonotone} gradient:
\[
f(x,y)\;=\;\tfrac{\rho L^2}{2}\,x^2\;+\;L\sqrt{1-\rho^2L^2}\,xy\;-\;\tfrac{\rho L^2}{2}\,y^2,
\]
The particular instance used in our experiment \ref{fig:comparison} specializes the parameters to
\[
\rho=-\tfrac{1}{3},\qquad L=1,
\]
so it becomes
\[ f(x,y) = -\tfrac{1}{6}x^2 + \tfrac{2\sqrt{2}}{3}xy + \tfrac{1}{6}y^2 \]
which is negative comonotonicity ( $\rho<0$ ), hence it lies outside the scope of our theory (in particular, the comonotonicity/monotonicity assumptions do not cover $\rho<0$). \\
As a result, the $\mathcal{O}(1/k^2)$-like decay observed for \textsc{GOMA} on this benchmark is purely empirical: we do not claim a guarantee and do not provide a proof in this regime. Any advantage over GOMA in this experiment should be interpreted as evidence of practical robustness rather than a theoretical rate. 
Formal analysis of GOMA under negative comonotonicity is left as an open problem.\\

\subsection{Stochastic example ~\S\ref{sec:exp_stochastic}}\label{appx:exp_stochastic}
\textbf{Comparison with DSEG} \\
\emph{Notation.} For consistency with \citep{NEURIPS2020_ba9a56ce}, this subsection 
adopts their notation: $X_t$ for the iterate, $V$ for the operator, 
$\hat V_t = V(X_t) + Z_t$ for the stochastic oracle, $\mathcal{F}_t$ for 
the natural filtration, $\beta$ for the Lipschitz constant, and 
$\gamma_t, \eta_t$ for the DSEG exploration and update step sizes. 
We write $\kappa_{\mathrm{H}}$ for the noise-growth parameter from their 
Assumption ($\mathbb{E}[\|Z_t\|^2 \mid \mathcal{F}_t] 
\le (\sigma + \kappa_{\mathrm{H}}\|X_t - x^\star\|)^2$), which is distinct 
from the $\kappa$ in our Assumption \ref{ass:bounded_v}.

Using the analysis of DSEG \citep{NEURIPS2020_ba9a56ce}, we can see why the method only guarantees convergence to an arbitrarily small neighborhood for general monotone problems.

Their main recursion implies
\begin{equation}\label{eq:egp_rec}
\mathbb{E}\!\left[\|X_{t+1}-x^\star\|^2 \mid \mathcal{F}_t\right]
\le (1+C_t\kappa_{\mathrm{H}}^2)\|X_t-x^\star\|^2
- \gamma_t\eta_t\!\left(1-\gamma_t^2L^2-8\gamma_t\eta_t\kappa_{\mathrm{H}}^2\right)\!\|V(X_t)\|^2
+ C_t\sigma^2 ,
\end{equation}
where $C_t$ depends on $\gamma_t,\eta_t$ and problem constants.

Under their error bound assumption,
\[
\|V(x)\|\ge \tau\,\mathrm{dist}(x,X^\star),
\]
a sufficient condition for contraction is
\begin{equation}\label{eq:egp_contr}
\gamma_t\eta_t\!\left(1-\gamma_t^2L^2-8\gamma_t\eta_t\kappa_{\mathrm{H}}^2\right)\tau^2 > C_t\kappa_{\mathrm{H}}^2 .
\end{equation}

Since $1-\gamma_t^2L^2-8\gamma_t\eta_t\kappa_{\mathrm{H}}^2 \le 1$ and $C_t \ge 4\eta_t^2$, a necessary condition for \eqref{eq:egp_contr} is
\begin{equation}\label{eq:egp_ratio}
\frac{\gamma_t}{\eta_t} > \frac{4\kappa_{\mathrm{H}}^2}{\tau^2}.
\end{equation}

In their implementation, the step sizes are chosen as $\gamma_t=\alpha/\beta_{\rm alg}$ and $\eta_t=\alpha$, where $\beta_{\rm alg}\ge1$ is a fixed step-size ratio. This yields
\[
\frac{\gamma_t}{\eta_t}=\frac{1}{\beta_{\rm alg}}\le1,
\]
which contradicts the necessary condition \eqref{eq:egp_ratio} whenever $4\kappa_{\mathrm{H}}^2/\tau^2>1$.

In high-dimensional games, $\tau\ll1$ is typically very small and $\kappa_{\mathrm{H}}=1-\tfrac1n\approx1$, so that $4\kappa_{\mathrm{H}}^2/\tau^2 \gg 1$. Consequently, the contraction condition can only be satisfied for vanishingly small step sizes, which leads to an arbitrarily slow convergence rate for DSEG (see Fig.~\ref{fig:exp_stochastic}).

\subsection{Details of experiments, bilinear game ($\kappa =1$)}\label{appx:exp_stoch_detail_kappa_1}

For fair comparison, we performed grid search over key hyperparameters for each method. RAIN \citep{chen2024near} is used in Case~I, where its bounded-variance assumption holds; RAIN++\citep{chen2024near} is used in Case~II, where neither method is formally in scope but RAIN++'s additional inner prox step empirically offers better robustness. For GOMA, we search over $\eta_{\mathrm{coef}} \in \{0.01, 0.1, 0.3, 0.5\}$ in $\eta_t = \eta_{\mathrm{coef}} \sqrt{\beta_t}$, selecting $\eta_{\mathrm{coef}} = 0.3$. For E-Halpern, we tune the step size factor $\eta_0 \in \{0.5, 1.0, 1.5, 2.0\}$ and batch size cap $B \in \{100, 200, 500\}$, selecting $\eta_0 = 1.5$ and $B = 500$. For RAIN, we search $\tau \in \{0.05, 0.1, 0.2\}$, $\lambda \in \{0.5, 0.7, 0.9\}$, and $\gamma \in \{0.0005, 0.001, 0.005\}$, selecting $\tau = 0.1$, $\lambda = 0.7$, $\gamma = 0.001$. For DSEG and FEG , we apply their recommended hyperparameters. For Nesterov, we search the step size $\alpha$ and constant momentum $\beta$ over a $10 \times 11$ grid (including negative momentum), selecting $\alpha = 0.005,\ \beta = -0.1$.

\subsection{Details of experiments, finite-sum game ($\kappa >1$)}\label{appx:exp_stoch_detail}

\label{app:hyperparams}     

\paragraph{Verification of Assumption~\ref{ass:bounded_v}.}
We first check that the finite-sum oracle of \cref{sec:exp_stoch_multiplicative} satisfies Assumption~\ref{ass:bounded_v} with explicit constants, so that the experiment is genuinely covered by Theorem~\ref{thm:last_iter_stoch_rho0}. Write $z=(\theta,\varphi)$, let the stochastic oracle be $F_i(z)=[\,b_i+A_i\varphi;\ -(A_i^\top\theta+c_i)\,]$ with $i$ sampled uniformly, and let $F(z)=\mathbb{E}_i[F_i(z)]=[\,\bar b+\bar A\varphi;\ -(\bar A^\top\theta+\bar c)\,]$ be the mean operator, with $z^\star$ the solution $F(z^\star)=0$. Define the deviations $\delta A_i:=A_i-\bar A$, $\delta b_i:=b_i-\bar b$, $\delta c_i:=c_i-\bar c$, and the constants
\[
V_0:=\frac1n\sum_{i=1}^n\bigl(\|\delta b_i\|^2+\|\delta c_i\|^2\bigr),
\qquad
\Lambda^2:=\frac1n\sum_{i=1}^n\|\delta A_i\|_{\mathrm{op}}^2 .
\]
Since $\mathbb{E}_i[F_i(z)-F(z)]=0$, the cross term vanishes and
\[
\mathbb{E}_i\|F_i(z)\|^2=\|F(z)\|^2+\frac1n\sum_{i=1}^n\|F_i(z)-F(z)\|^2 .
\]
Each deviation $F_i(z)-F(z)=[\,\delta b_i+\delta A_i\varphi;\ -(\delta A_i^\top\theta+\delta c_i)\,]$ is affine in $z$, so by $\|u+v\|^2\le2\|u\|^2+2\|v\|^2$ and $\|\delta A_i w\|\le\|\delta A_i\|_{\mathrm{op}}\|w\|$,
\[
\frac1n\sum_{i=1}^n\|F_i(z)-F(z)\|^2
\;\le\; 2V_0+2\Lambda^2\|z\|^2 .
\]
Because $F$ is affine and $F(z^\star)=0$, we have $F(z)=[\,\bar A(\varphi-\varphi^\star);\ -\bar A^\top(\theta-\theta^\star)\,]$, hence $\|F(z)\|^2\ge\sigma_{\min}(\bar A)^2\,\|z-z^\star\|^2$ whenever $\bar A$ is nonsingular. Combining with $\|z\|^2\le2\|z-z^\star\|^2+2\|z^\star\|^2$ gives
\[
\mathbb{E}_i\|F_i(z)\|^2
\;\le\;
\underbrace{2V_0+4\Lambda^2\|z^\star\|^2}_{\sigma^2}
\;+\;
\underbrace{\Bigl(1+\tfrac{4\Lambda^2}{\sigma_{\min}(\bar A)^2}\Bigr)}_{\kappa}\|F(z)\|^2 ,
\]
which is exactly Assumption~\ref{ass:bounded_v} with $\kappa>1$ (state-dependent noise) and finite $\sigma^2$. For the specific construction, $A_i=\mathrm{diag}(0,\ldots,\lambda_i,\ldots,0)$ gives $\bar A=\tfrac1n\mathrm{diag}(\lambda_1,\ldots,\lambda_n)$, so $\sigma_{\min}(\bar A)=\tau/n>0$ and $\bar A$ is full rank; thus the construction lies in the $\kappa>1$ regime of Theorem~\ref{thm:last_iter_stoch_rho0}, whereas it violates the bounded-variance assumption ($\kappa=1$) of FEG, E-Halpern, and RAIN++.

\paragraph{Hyperparameters.}

All methods are tuned via grid search to minimize the final $\|F(z_k)\|^2$.

\paragraph{GOMA.}
$\beta_k = 1/(k+2)$, $\eta_k = c\sqrt{\beta_k}$ with $c \in \{1.0, 1.2, 1.5, 2.0, 2.3, 2.5\}$. Best: $c = 1.5$.

\paragraph{DSEG.}
$\eta_k^{(1)} = 0.1/(k+1)^{0.1}$, $\eta_k^{(2)} = 0.1/(k+1)^{0.9}$. according to \cite{NEURIPS2020_ba9a56ce}
\paragraph{FEG.}
$\alpha = 1$, $\rho = 0$, $\beta_k = 1/(k+1)$. according to \cite{lee2021fast}

\paragraph{RAIN++.}
$\mu \in \{0.5, 1, 2, 4\}$, $\tau \in \{0.05, 0.1, 0.2, 0.3, 0.5\}$, $\alpha \in \{0.1, 0.2, 0.4, 0.6\}$, $\gamma \in \{10^{-4}, 10^{-3}, 10^{-2}, 10^{-1}\}$, with 3 inner steps. Best: $(\mu, \tau, \alpha, \gamma) = (0.5, 0.1, 0.1, 10^{-4})$.

\paragraph{E-Halpern with PAGE.}
$\lambda_k = 1/(k+1)$, $p_k = 2/(k+1)$, $s_1 = 6$. Search: $\eta_0 \in \{0.01, 0.05, 0.1, 0.2, 0.5\}$, $M \in \{1, 3, 9, 27\} \cdot L^2$, $s_2 \in \{1, 2\}$. Best: $(\eta_0, M, s_1, s_2) = (0.1, 27L^2, 6, 1)$.

\paragraph{Nesterov.}
$y_k = z_k + \beta(z_k - z_{k-1})$, $z_{k+1} = y_k - \alpha \widehat{F}(y_k)$,
with $\alpha \in \{0.01, \ldots, 3.0\}$ and $\beta \in \{-0.9, \ldots, 0.9\}$.
Best: $(\alpha, \beta) = (0.01, -0.9)$

\subsection{Extra experiment: Monotone QP Lagrangian}
\label{appx:exp_monotone_qp}

\paragraph{Setup.}
We consider the Lagrangian of a linearly constrained quadratic minimization problem from \cite{pmlr-v139-yoon21d}:
\begin{equation}
    L(\mathbf{x}, \mathbf{y}) = \tfrac{1}{2}\mathbf{x}^\top \mathbf{H}\mathbf{x} - \mathbf{h}^\top \mathbf{x} - \langle \mathbf{A}\mathbf{x} - \mathbf{b}, \mathbf{y} \rangle,
\end{equation}
where $\mathbf{x}, \mathbf{y} \in \mathbb{R}^n$, $\mathbf{A} \in \mathbb{R}^{n \times n}$, $\mathbf{b} \in \mathbb{R}^n$, $\mathbf{H} \in \mathbb{R}^{n \times n}$ is positive semidefinite, and $\mathbf{h} \in \mathbb{R}^n$. This saddle function is convex--concave and smooth (due to the quadratic term $\tfrac{1}{2}\mathbf{x}^\top \mathbf{H}\mathbf{x}$); its saddle operator is monotone and $1$-Lipschitz.

We use the specific construction from  \cite{pmlr-v139-yoon21d}:
\begin{equation}
    \mathbf{A} = \frac{1}{4}\begin{bmatrix} -1 & 1 & & \\ & \ddots & \ddots & \\ & & -1 & 1 \\ & & & 1 \end{bmatrix} \in \mathbb{R}^{n \times n}, \quad
    \mathbf{b} = \frac{1}{4}\begin{bmatrix} 1 \\ 1 \\ \vdots \\ 1 \\ 1 \end{bmatrix} \in \mathbb{R}^{n}, \quad
    \mathbf{h} = \frac{1}{4}\begin{bmatrix} 0 \\ 0 \\ \vdots \\ 0 \\ 1 \end{bmatrix} \in \mathbb{R}^{n},
\end{equation}
and $\mathbf{H} = 2\mathbf{A}^\top\mathbf{A}$. \cite{pmlr-v139-yoon21d} shows that $\|\mathbf{A}\| \leq \tfrac{1}{2}$, which implies $\|\mathbf{H}\| \leq \tfrac{1}{2}$. Therefore this is a $1$-smooth saddle problem. We set $n=200$.

\paragraph{Methods.}
We compare the same set of methods as in the negative comonotone experiments.

\paragraph{Hyperparameter selection.}
\begin{itemize}
    \item \textbf{EAG-C / EAG-V}: Since this experiment is from \cite{pmlr-v139-yoon21d}, we use their reported parameters: $\alpha = 0.1265/L$ for EAG-C and $\alpha_0 = 0.618/L$ for EAG-V, with $\beta_k = 1/(k+2)$.
    \item \textbf{FEG} \cite{lee2021fast}: Following the FEG paper, we set $\alpha = 1/L$ with $\beta_k = 1/(k+1)$, which is the same choice used in both monotone and negative comonotone settings.
    \item \textbf{Anchored Popov}: $\alpha = 0.9/L$ with $\beta_k = 1/(k+1)$, selected via grid search over $\alpha = c/L$ with $c \in \{0.3, 0.4, 0.5, 0.618, 0.7, 0.8, 0.9, 1.0, 1.1, 1.2\}$.
    \item \textbf{EG}: $\alpha = 0.5/L$, the same as in \cite{pmlr-v139-yoon21d}.
    \item \textbf{DSEG} \cite{NEURIPS2020_ba9a56ce}: Since DSEG is primarily designed for the stochastic setting, we use the same parameters as in our negative comonotone experiments, following the setting of the FEG paper.
    \item \textbf{Nesterov}: $y_k = z_k + \beta(z_k - z_{k-1})$, $z_{k+1} = y_k - \alpha\,F(y_k)$, selected via grid search over $\alpha \in \{0.01, 0.05, 0.1, 0.3, 0.5, 0.7, 1.0, 1.3, 1.5, 2.0\}/L$ and constant momentum $\beta \in \{-0.9, -0.7, \ldots, 0.9\}$ (including negative momentum). Best: $\alpha = 1/L$, $\beta = -0.1$. The scheduled momentum $\beta_k = k/(k+3)$ diverged at every step size tested.
    \item \textbf{GOMA}: $\alpha = 1.25/L$ with $\eta_k = \alpha$, $\gamma_k = \alpha(1 - \beta_k)$, and $\beta_k = 1/(k+1)$, selected via grid search over $\alpha \in \{0.3, 0.4, 0.5, 0.618, 0.7, 0.8, 0.9, 1.0, 1.1, 1.2, 1.25, 1.3\}/L$.
\end{itemize}

\paragraph{Results.}
As shown in Figure~\ref{fig:monotone_qp}, all anchoring-based methods achieve the optimal $\mathcal{O}(1/k^2)$ rate, clearly separating from EG, DSEG, and Nesterov. GOMA outperforms the two-gradient-call methods (EAG-C, EAG-V, FEG) by a constant factor, which is expected given its single gradient call per iteration. Compared to Anchored Popov (which also uses a single call), GOMA achieves a better constant thanks to its two-time-scale structure that decouples the exploration and update step sizes.

\begin{figure}[h]
    \centering
    \includegraphics[width=0.6\textwidth]{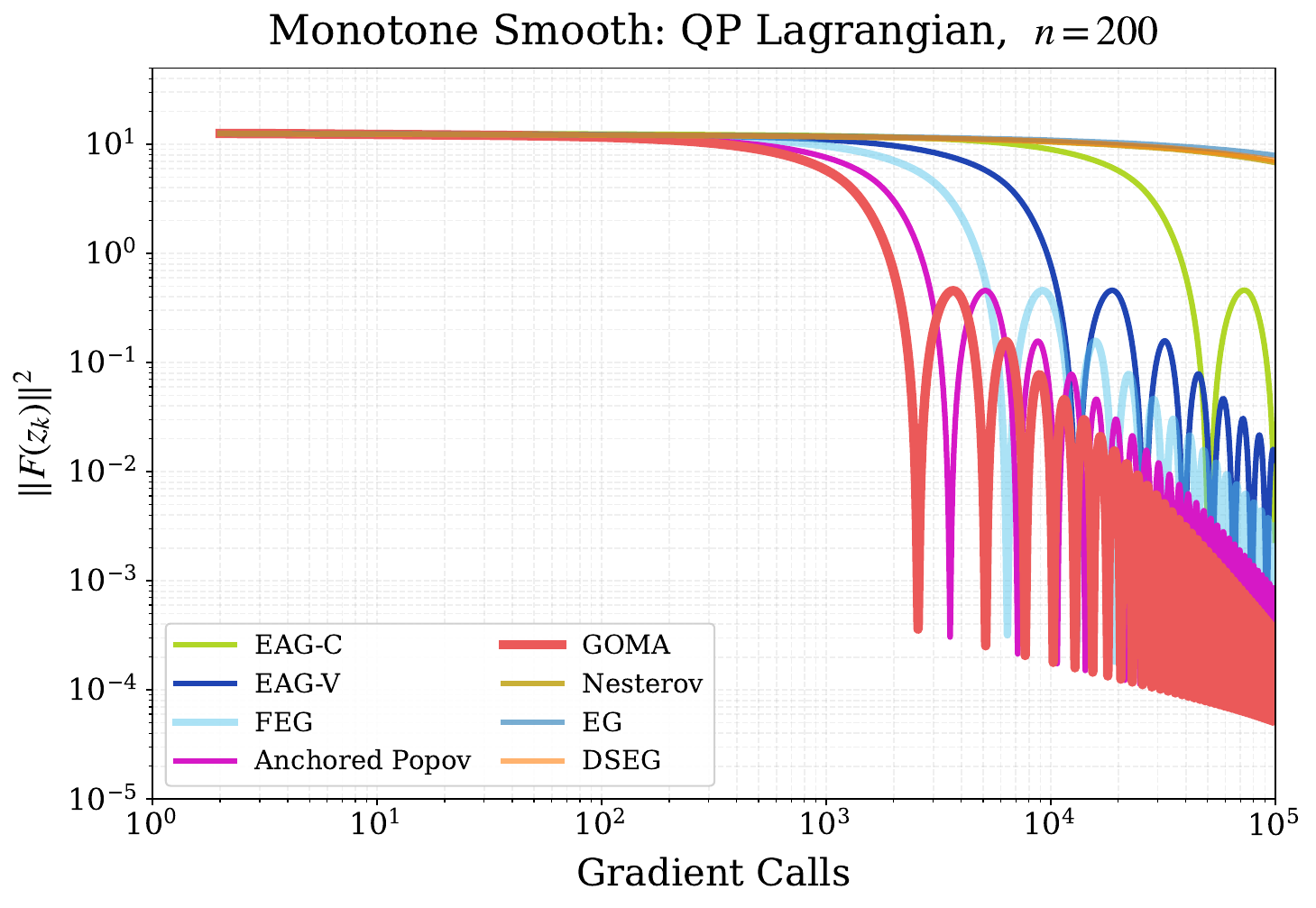}
    \caption{Monotone QP Lagrangian ($n=200$, $L=1$). Anchoring-based methods (EAG-C/V, FEG, Anchored Popov, GOMA) achieve the $\mathcal{O}(1/k^2)$ rate on $\|F(z_k)\|^2$, while EG, DSEG, and Nesterov exhibit a substantially slower decay and largely overlap near the top of the plot. GOMA (red) attains the lowest final $\|F(z_k)\|^2$.}
    \label{fig:monotone_qp}
\end{figure}




\end{document}